\documentclass[12pt]{article}

\usepackage{amsmath}
\usepackage[natbibapa]{apacite}

\usepackage{amssymb,color}
\usepackage{appendix}
\usepackage{amsthm}
\usepackage{enumerate}
\usepackage{hyperref}
\usepackage{bbm} 
\DeclareMathAlphabet{\mymathbb}{U}{BOONDOX-ds}{m}{n}
\usepackage{dsfont}

\usepackage{mathtools} 
\mathtoolsset{showonlyrefs=true}  
\usepackage{caption}
\usepackage[margin=0.7in]{geometry}
\numberwithin{equation}{section}
\usepackage{comment}

\usepackage{titling}
\newtheorem{thm}{Theorem}[section]

\newtheorem{defi}{Definition}[section]
\newtheorem{assume}{Assumption}[section]
\newtheorem{prop}[thm]{Proposition}

\newtheorem{lemma}[thm]{Lemma}
\newtheorem{ex}{Example}[section]
\newtheorem{remark}{Remark}[section]

\begin{document}

\author{Kiseop Lee\thanks{Department of Statistics, Purdue University, West Lafayette, IN, United States, kiseop@purdue.edu}, Seongje Lim \thanks{Department of Mathematical Sciences, Seoul National University, 1, Gwanak-ro, Gwanak-gu, Seoul, Republic of Korea, tjdwpdla@snu.ac.kr} and Hyungbin Park \thanks{Department of Mathematical Sciences, Seoul National University, 1, Gwanak-ro, Gwanak-gu, Seoul, Republic of Korea, hyungbin@snu.ac.kr, hyungbin2015@gmail.com}   }

\title{Option pricing under path-dependent stock models}
\maketitle
\abstract{ 
	
This paper studies how to price and hedge options under stock models given as  a path-dependent SDE solution.
When 
the path-dependent SDE coefficients have  Fr\'{e}chet derivatives,
an option price is differentiable with respect to time and the path, and is given as a solution to the path-dependent PDE. 
This can be regarded as  a path-dependent version of the Feynman-Kac formula.
As a byproduct, we obtain the differentiability of  path-dependent SDE solutions and  the SDE representation of their derivatives. 
In addition,  we provide formulas for Greeks with 
path-dependent coefficient perturbations.
A stock model having coefficients with time integration forms of paths is covered as an example. 
}

\section{Introduction}
\label{intro}

\subsection{Overview}

This paper aims to investigate option prices under a stock  model given as a solution to an SDE
 with path-dependent coefficients having  Fr\'{e}chet derivatives.  
We show that 
an option price is differentiable with respect to time and   path, and is given as a solution to a path-dependent PDE. 
Using this result, 
PDE representations and hedging portfolios of options under path-dependent stock models are investigated.
In addition, we derive option Greeks with path-dependent coefficient perturbations.

One of main purposes of this paper is to
provide a PDE representation of 
option prices under path-dependent stock models.
This is a path-dependent version of the Feynman-Kac formula and can be regarded as  extended results of  \cite{PengWang2016bsde}.
We consider a market model  in which the logarithm of stock price follows a path-dependent SDE. The stock price process $S=(S_t)_{t\ge0}$ is given as $S_{t} = S_{0}e^{X_{t}}$, where
\begin{align} \label{model risk-neutrala}
dX_t = b_t(X)\,dt + \sigma_t(X)\,dW_t, \, X_0=0,
\end{align} 
for a non-anticipative functional $b, \sigma$
under the risk-neutral measure.
A path-dependent option price is $v_t(X_t) = e^{-r(T-t)}u_T(X_t)$ where $u_t$ is defined as 
\begin{align} \label{model option price1}
u_t(X_t):= e^{-r(T-t)} \mathbb{E}\left[g(X_{T}) | \mathcal{F}_{t} \right],\;0\le t\le T\,,
\end{align}
for a constant short rate $r$ and a non-anticipative functional $g.$
We show that $u$ has first- and second- order vertical derivatives and a horizontal derivative, denoted as $D_{x}u,$ $D_{xx}u$ and $D_tu,$ respectively, and $u$ satisfies a path-dependent PDE:
\begin{equation}
\begin{aligned} \label{main PPDE_1}
D_{t}u_t(\gamma_{t}) &+ b_t(\gamma_{t}) D_{x}u_t(\gamma_{t}) + \frac{1}{2}\sigma_t^{2}(\gamma_{t}) D_{xx}u_t(\gamma_{t}) = 0, \quad &&\textnormal{for} \quad &&\gamma_{t} \in D([0, t], \mathbb{R}^d)  , \\
u_T(\gamma_{T})      &= g(\gamma_{T}), \quad &&\textnormal{for} \quad &&\gamma_{T} \in D([0, T], \mathbb{R}^d).
\end{aligned}
\end{equation}
The converse  also holds true. The precise statement is given in Theorem \ref{main theorem}.

\cite{PengWang2016bsde} studied a nonlinear Feynman-Kac formula for non Markovian BSDEs as in \eqref{main PPDE_1} with a standard Brownian motion $X_t=W_t$ (i.e. $b \equiv 0, \sigma \equiv 1$).
It is not straightforward to extend their results to an Ito process $X$ having path dependent drift and volatility terms. The reason is as follows.
As a main technique, they used 
the property that a Brownian motion with initial value $x_0$ is the parallel translation of a standard Brownian motion by $x_0$. 
More precisely, the Brownian increment is given as $W^{\gamma_t + \delta_t} - W^{\gamma_t} =\delta_t$ where $W^{\gamma}(s) = \gamma(s)\mathbbm{1}_{[0,t]}(s) + (\gamma(t) + W(s) - W(t))\mathbbm{1}_{( t, T ]}(u)$ with a standard Brownain motion $W$. 
However, we cannot expect that this property holds for a path-dependent SDE solution, thus it cannot be extended directly to a path-dependent SDE solution.
This paper is an extension of the above paper to an Ito process $X$ having path dependent drift and volatility.
To overcome the technique in Peng and Wang, we adopt Fr\'{e}chet derivative to 
estimate the term as $X^{\gamma_t + \delta_t} - X^{\gamma_t}$.

As another application, 
we conduct a sensitivity analysis of option prices with respect to the coefficient perturbations of stock. 
For the perturbed coefficients
$b^\epsilon$ and $\sigma^\epsilon$  with the perturbation parameter $\epsilon,$
let $X^\epsilon$ be a solution to  
\begin{equation} \label{greek stock model1}
\begin{aligned}
dX^{\epsilon}_{t} &= b^{\epsilon}_t(X^{\epsilon})\,dt + \sigma^{\epsilon}_{t}(X^{\epsilon})\,dW_{t}, \quad X^{\epsilon}_{0} =0.
\end{aligned}
\end{equation}
At time $0,$ the perturbed  option price is 
\begin{equation}
  \begin{aligned}
    v_0^\epsilon(X_0)= e^{-r^\epsilon T} \mathbb{E}\left[g(X_{T}^\epsilon)  \right]  
  \end{aligned}
\end{equation}
for the perturbed short rate $r^\epsilon.$
We show that 
\begin{equation} 
\begin{aligned} 
&\lim_{\epsilon \rightarrow 0} \frac{1}{\epsilon} (v_0^\epsilon(X_0)-v_0(X_0))\\
=\;&	 -u_0(X_{0})e^{-rT}\dot{r}T
	- e^{-rT} \mathbb{E}\Big[\int_{0}^{T} D_{x}u_s(X_s) \dot{b}_{s}(X) +  D_{xx}u_s(X_s) \dot{\sigma}_{s}(X)\sigma_{s}(X) \,ds\Big] 
	\end{aligned}
\end{equation}	
where
$\dot{b},\dot{\sigma},\dot{r}$ are the partial derivatives of $b,\sigma,r$, respectively, of $\epsilon$ evaluated at $\epsilon=0$.  
Precise assumptions regarding perturbations and the results are given in Section~\ref{subsec:G}. 

The rest of this paper is structured as follows. 
The related literature is reviewed in
Section \ref{relatedarticle}.
We explain the basic concepts of functional It\^{o} calculus
in Section \ref{sec:2}.
In Section \ref{sec:d}, we
investigate the differentiability with respect to the initial path of path-dependent SDEs and state the main result of this paper.
In Section \ref{sec:price}, 
our main results are applied to option pricing theory for an exponential path-dependent stock price model.
We show that the option price is a $\mathbb{C}^{1, 2}$ solution to a path-dependent partial differential equation 
and provide sensitivity formulas. 
Section \ref{sec:ex} presents an example for a stock price model.
Finally, the last section summarizes
the paper. The proofs of the main results are presented
in the appendices.

\subsection{Literature review}\label{relatedarticle}

Recently, path-dependent SDEs have been actively researched. 
\cite{Dupire2019} extended Itô calculus to functionals of the current path of a process. To obtain an Itô formula, they expressed the differential of the functional in terms of partial derivatives. Moreover, they developed an extension of the Feynman-Kac formula to the functional case and an explicit expression of the integrand in the martingale representation theorem. \cite{Cont2010a} derived a change of variable formula for non-anticipative functionals. Their results led to functional extensions of the Itô formula for a large class of stochastic processes.
\cite{Cont2010} showed that the functional derivative admits a suitable extension to the space of square-integrable martingales. 
This extension defines a weak derivative viewed as a non-anticipative lifting of the Malliavin derivative.
These results led to a constructive martingale representation formula for Itô processes.

The Feynman-Kac formula has been extended to non-Markovian cases. 
\cite{Peng2010}
studied a type of path-dependent quasi-linear parabolic PDEs.
These PDEs are formulated through a classical BSDE in which the terminal values and generators
are general functions of Brownian motion paths.
As a result, they established the nonlinear Feynman-Kac formula for a general non-Markovian BSDE.
\cite{Ekren2014} proposed a notion of viscosity solutions for path-dependent semi-linear parabolic PDEs. 
This can be viewed as viscosity solutions of non-Markovian backward SDEs. They proved the existence, uniqueness, stability, and comparison principle for the viscosity solutions. 
\cite{bouchard2021approximate} introduced a notion of approximate viscosity solution for nonlinear path-dependent PDEs including the Hamilton-Jacobi-Bellman type equations. They investigated the existence, stability, comparison principle and regularity of the solution to the PDEs.

Estimating Greeks 
is an important topic in mathematical finance. 
\cite{Fournie1999} presents a probabilistic method for the numerical computations of Greeks in finance. 
They applied the integration-by-parts formula developed in the Malliavin calculus to exotic European options in the framework of the Black and Scholes model.  Their method was compared to the Monte Carlo finite difference approach and turned out to be efficient in the case of discontinuous payoff functionals.

\section{Functional It\^{o} calculus}\label{sec:2}
In this section, we introduce basic notions of functional It\^{o} calculus. We fix a probability space $(\Omega, \mathcal{F}, (\mathcal{F}_{t})_{t \ge 0}, \mathbb{P})$ and  time maturity $T \in \mathbb{R}_{+}$. Let $D([0, T], \mathbb{R}^d)$ be a collection of all functions with c\`{a}dl\`{a}g paths from $[0, t]$ to $\mathbb{R}^d.$ We also define $\tilde{D}^{d} = \bigcup_{t \in [0, T]} D([0, t], \mathbb{R}^d)$. 
We define $\mathbb{D}^d $ as a collection of all $d$-dimensional stochastic processes with c\`{a}dl\`{a}g paths. (i.e., $X \in \mathbb{D}^d$ means that $X$ is a stochastic process defined on $[0, T]$ and $X(\omega) \in D([0, T], \mathbb{R}^d)$ for $\omega\in \Omega$).

We write $\eta^{t}(\cdot) = \eta(\cdot \wedge t) \in D([0, T], \mathbb{R}^d)$ to denote the stopped path of $\eta$ at time $t \in [0, T]$ for $\eta \in D([0, T], \mathbb{R}^d)$. We denote the value of path $\gamma_{t} \in D([0, t] , \mathbb{R}^d)$ at time $s \in [0, t]$ as $\gamma_{t}(s) \in \mathbb{R}^{d}$.

For each $t \in [0, T]$, let $\gamma_{t} \in D([0, t], \mathbb{R}^d)$. We introduce several notations
\begin{align}
  \left\lVert \gamma_{t} \right\rVert_{s} & := \sup_{u \in [0, s]}\left\lVert \gamma(u) \right\rVert :=  \sup_{u \in [0, s]}\max_{i = 1, \cdots d}\left\lvert (\gamma(u)_1, \gamma(u)_2, \cdots \gamma(u)_d) \right\rvert, \\ 
  \quad \left\lVert \gamma_{t} \right\rVert_{s_0, s_1} & := \sup_{u \in [s_0, s_1]}\left\lVert \gamma(u) \right\rVert, 
  \quad \left\lVert \gamma_{t} \right\rVert := \left\lVert \gamma_{t} \right\rVert_{t},
\end{align}
for $s, s_0, s_1 \in [0, t]$. The map $ \lVert \cdot \rVert = \lVert \cdot \rVert_{t}$ is regarded as a norm in $D([0, t], \mathbb{R}^d)$. We use the same notation in $\mathbb{D}^{d}$, 
\begin{align}\label{ftnl notation}
  \left\lVert X \right\rVert_{s} := \sup_{u \in [0, s]}\left\lVert X(u) \right\rVert, 
  \; \left\lVert X \right\rVert_{s_0, s_1} := \sup_{u \in [s_0, s_1]}\left\lVert X(u) \right\rVert, 
  \; \left\lVert X \right\rVert := \left\lVert X \right\rVert_{T}.
\end{align}

Let $F : [0, T] \times D([0, T], \mathbb{R}^d) \rightarrow \mathbb{R} $. We say that $F$ is a non-anticipative functional if for any $\eta, \gamma \in D([0, t], \mathbb{R}^d)$ satisfying $\eta^{s} = \gamma^{s}$, the equality $F(s, \eta) = F (s, \gamma)$ holds. This means that the value of $F(s, \gamma)$ depends only on the $[0, s]$ part of path $\gamma$.
The non-anticipative functional $F$ can be regarded as an operator $\tilde{F}$ from $\tilde{D}^{d}$ to $\mathbb{R}$ by defining $\tilde{F}(\gamma) = F(t, \gamma^{T})$ for $\gamma \in D([0, t], \mathbb{R}^d)$ where $\gamma^{T}(s) = \gamma(s) \mathds{1}_{[0, t]}(s) + \gamma(t)\mathds{1}_{[t, T]}(s)$.
We can denote $F_{t}(\eta) = F(t, \eta) \in \mathbb{R}$ for a non-anticipative functional $F$ if there is no ambiguity.
Next, we introduce the notions of horizontal and vertical derivatives, which indicate the sensitivity of a functional to each perturbation of a stopped path $(t, \eta^{t})$.
\begin{defi}\label{ftnl ito defi}
  Let $F :[0, T] \times D([0, T], \mathbb{R}^d) \rightarrow \mathbb{R}$ be a non-anticipative functional.
  \begin{enumerate}
    \item [(i)] We say that $F$ is left continuous if
    \begin{enumerate}
      \item  for each $t \in [0, T]$, $F(t, \cdot) : (D([0, T], \mathbb{R}^d), \lVert \cdot \rVert_{T} ) \rightarrow \mathbb{R}$ is continuous,
      \item  for any $(t, \eta) \in [0, T] \times D([0, T], \mathbb{R}^d)$ and $\epsilon > 0$, there exists $\delta > 0$ if $(t', \eta') \in [0, T] \times D([0, T], \mathbb{R}^d)$ satisfy $t' < t$ and $\lVert \eta - \eta' \rVert _{t'} + \lvert t - t' \rvert < \delta$, then $\lvert F(t, \eta) - F(t', \eta') \rvert < \epsilon$.
    \end{enumerate}
    \item [(ii)] We say that $F$ is boundedness-preserving if for any $K \in \mathbb{R}_{+}$ and $t < T$, there exists $C > 0$ such that if $s \le t$ and $\lVert \eta \rVert_{s} < K$ for $\eta \in D([0, T], \mathbb{R})$, then $\lvert F(s, \eta) \rvert \le C$.
    \item [(iii)] $F$ is said to be horizontally differentiable at $(t, \eta)\in[0, T] \times D([0, T], \mathbb{R}^d)$ if the limit
    \begin{align}
      D_tF(t, \eta) = \lim_{\delta \rightarrow 0+} \frac{\left\lvert F(t+\delta, \eta^{t} ) - F(t, \eta^{t}) \right\rvert }{\delta}  
    \end{align}
    exists. We call $D_tF(t, \eta)$ the horizontal derivative $D_tF(t, \eta)$ of $F$ at $(t, \eta)$.
    \item [(iv)] $F$ is said to be vertically differentiable at $(t, \eta)\in[0, T] \times D([0, T], \mathbb{R}^d)$
    if the limit
    \begin{align}
      \partial_i F(t, \eta) = \lim_{\delta \rightarrow 0+ }\frac{\left\lvert F(t, \eta^{t} + \delta e_{i} \mathds{1}_{[t, T]}) - F(t, \eta^{t} ) \right\rvert }{\delta}
    \end{align}
    exists for $i = 1, \cdots , d$ where $e_{i}$ is i-th basis. \\
     We call $D_x F(t, \eta) := (\partial_{1}F(t, \eta), \cdots, \partial_{1}F(t, \eta)) \in \mathbb{R}^{1\times d}$ the vertical derivative $D_x F(t, \eta)$ of $F$ at $(t, \eta)$.
  \end{enumerate}
\end{defi}

If $F$ is horizontally differentiable at all $(t, \eta)$, the map $DF : (t, \eta) \mapsto DF(t, \eta)$ is a non-anticipative functional. The same holds for the vertical derivative. Similarly, we can define the second-order vertical derivative $D_{xx}F = D_x(D_xF)$ of $F$.
\begin{defi} \label{ftnl regular def}
  We say $F$ is a regular functional if there exist $D_tF, D_xF, D_{xx}F$ for all $(t, \eta) \in [0, T] \times D([0, T], \mathbb{R}^d)$, and we have
  \begin{enumerate}
    \item [(i)] For each $t \in [0, T]$, $D_tF(t, \cdot ) : (D([0, T], \mathbb{R}^d),\lVert \cdot \rVert_{T} ) \rightarrow \mathbb{R}$ is continuous,
    \item [(ii)] $D_xF, D_{xx}F$ are left continuous,
    \item [(iii)] and $D_tF, D_xF, D_{xx}F$ are boundedness-preserving.
  \end{enumerate}
\end{defi}

From \cite{Cont2016FuntioncalItoCalculus}, we get the continuous version of the functional It\^{o} formula.
\begin{thm}  \label{ftnl Ito formula}
  (Functional It\^{o} formula : continuous case) Let $X$ be a continuous $d$-dimensional semimartingale and $F$ be a regular functional. For any $t \in [0, T)$,
  \begin{align}
    F(t, X_{t}) - F(0, X_{0}) &= \int_{0}^{t}D_t F(u, X_{u})\,du + \int_{0}^{t} D_x F(u, X_{u})\,dX(u) \\ 
    &+\frac{1}{2}\int_{0}^{t} \textnormal{tr}( D_{xx}F(u, X_{u})\,d[X](u) ) \quad a.s.
  \end{align}
\end{thm}

\section{Differentiability of stochastic flow}\label{sec:d}
In this section, we investigate the differentiability with respect to the initial path of a path-dependent SDE.
The result is the path-dependent version of Theorem~V.39 of \cite{protter2005stochastic}. \cite{protter2005stochastic} deals with the differentiability with respect to an initial data point of an SDE when the coefficients of the SDE are a function depending on a finite point. The condition of this theorem is that the coefficient function of the SDE is differentiable. Therefore, we introduce the Fr\'{e}chet derivative of a non-anticipative functional corresponding to the differentiability of the function.

Let $V$ and $W$ be two normed vector spaces. We denote the family of linear bounded operators from $V$ to $W$ with the operator norm by $L(V, W)$. The Fr\'{e}chet derivative of an operator between two normed vector spaces can be defined. 

\begin{defi} \label{main Fre def}
  Let $V$, $W$ be vector spaces with norm $||\cdot||_V$ and $||\cdot||_W$.
  \begin{enumerate}
    \item [(i)] We say that an operator $F : V \rightarrow W $ has a Fr\'{e}chet derivative at $v \in V$ if there exists a bounded linear operator $DF(v) : V \rightarrow W $ such that
    \begin{align}
      \lim_{\left\lVert u \right\rVert_{V} \rightarrow 0} \frac{\left\lVert F(v + u) - F(v) - DF(v)(u) \right\rVert_{W}}{\left\lVert u \right\rVert_{V}} = 0.
    \end{align}
  
    \item [(ii)] Suppose that an operator $F : V \rightarrow W$ has a Fr\'{e}chet derivative for all $v \in V$. We say that $F$ has a second-order Fr\'{e}chet derivative at $v \in V$ if $DF : V \rightarrow L(V, W)$ has a Fr\'{e}chet derivative at $v \in V$. This means that there exists a bounded linear operator $D^{2}F(v) : V \rightarrow L(V, W)$ such that
    \begin{align}
      \lim_{\left\lVert u \right\rVert_{V} \rightarrow 0} \frac{\left\lVert DF(v + u) - DF(v) - D^{2}F(v)(u) \right\rVert_{L(V, W) }}{\left\lVert u \right\rVert_{V} } = 0.
    \end{align}
  \end{enumerate}
  \end{defi}

An operator $F : V \rightarrow W $ is said to have a Fr\'{e}chet derivative if $F$ has a Fr\'{e}chet derivative for all $v \in V$. We say that the operator $F$ has a second Fr\'{e}chet derivative as in a similar way. For simplicity, we may sometimes denote $D^{2}F(v)(w)$ as $D^{2}F(v,w)$. 

Before applying the above definitions to non-anticipative functionals, we define some notions following \cite{protter2005stochastic}. Fix a maturity $T \in \mathbb{R}_{+}$. A non-anticipative functional $F : [0, T] \times D([0, T], \mathbb{R}^d) \rightarrow \mathbb{R}^{N}$ is functional Lipschitz continuous if for each $t \in [0, T]$, there exists a constant $K(t)$ such that 
\begin{align}
  \left\lVert F_t(\eta) - F_t(\delta) \right\rVert  \le K(t) \left\lVert \eta - \delta \right\rVert_t 
\end{align}
for any $\eta, \delta \in D([0, T], \mathbb{R}^d)$.
There are some conditions for a non-anticipative functional $F$ to apply the Fr\'{e}chet derivative.
\begin{assume} \label{main F ass}
  Let $ F : [0, T] \times D([0, T], \mathbb{R}^d) \rightarrow \mathbb{R}$ be a non-anticipative functional. The functional $F$ satisfies the following: 
   \begin{enumerate}
    \item [(i)] for each $\eta \in D([0, T], \mathbb{R}^d)$, the map $F(\eta) : [0, T] \rightarrow \mathbb{R}^{N}$ is a c\`{a}dl\`{a}g path,
    \item [(ii)] there exists an increasing sequence of open sets $A_k$ such that $\bigcup_k A_k = D([0, T], \mathbb{R}^d)$ and $F$ is functional Lipschitz continuous on each $A_k$.
   \end{enumerate}
\end{assume}
The first condition means that we can regard the functional $F : [0, T] \times D([0, T], \mathbb{R}^d) \rightarrow \mathbb{R}^{N}$ as the operator $F : D([0, T], \mathbb{R}^d) \rightarrow D([0, T], \mathbb{R}^{N} )$. The second is a necessary condition for the existence of the Fr\'{e}chet derivative and the solution to the SDE. We refer to satisfying the second condition as a local functional Lipschitz continuity. 

Let $Z$ be a $N$-dimensional semimartingale with $Z_{0} = 0$ and $F : D([0, T], \mathbb{R}^d) \rightarrow D([0, T], \mathbb{R}^{d \times N} )$ be a non-anticipative functional satisfying Assumption \ref{main F ass} and having the Fr\'{e}chet derivative $DF$. Consider a solution $(X^{x},DX^{x})$ to the $N$-dimensional SDE with a non-anticipative functional coefficient 
\begin{align}
  X^{x}_{t}  & = x + \int_{0}^{t}F_{s-}(X^{x})\,dZ_{s}, \label{main Fre SDEX}           \\
  DX^{x}_{t} & = 1 + \int_{0}^{t}DF_{s-}(X^{x})(DX^{x})\,dZ_{s}, \label{main Fre sdeDX}
\end{align}
for $x \in \mathbb{R}^{d}$. Since $F$ and $DF$ are locally functional Lipschitz continuous, we know that the solutions of \eqref{main Fre SDEX} and \eqref{main Fre sdeDX} uniquely exist.

The following theorems deal with the differentiability of a solution to an SDE with a non-anticipative functional coefficient. Under the good condition of the coefficient $F$, a solution to the SDE can be differentiable, and the differential is written as a solution to an SDE. This means we can apply the property of SDE to the differential. We can regard this proposition as a non-anticipative functional version of Theorem~V.39 of \cite{protter2005stochastic}.

\begin{prop} \label{main Fre prop1}
  Let $F : D([0, T], \mathbb{R}^d) \rightarrow D([0, T], \mathbb{R}^{d \times N} )$ be a non-anticipative functional with Assumption~\ref{main F ass}. Assume that $F$ has the Fr\'{e}chet derivative $DF$ which is locally Lipschitz continuous in the operator norm, i.e., for any $\eta \in D([0, T], \mathbb{R}^d)$, there exist constant $\epsilon, K(\eta) > 0$ such that if $\lVert \eta - \tilde{\eta} \rVert < \epsilon$ then
 \begin{align}\label{main Fre prop1 ass}
    \left\lVert DF(\eta)(\delta) - DF(\tilde{\eta})(\delta) \right\rVert \le K(\eta) \left\lVert \eta-\tilde{\eta} \right\rVert \left\lVert \delta \right\rVert,
  \end{align}
holds for all $\delta \in D([0, T], \mathbb{R}^d)$. Suppose that $X$ is a solution to SDE \eqref{main Fre SDEX}. Then, the following conditions are satisfied.
  \begin{enumerate}
    \item [(i)] For almost everywhere $\omega \in \Omega$, there exists a function $X^{x}(t, \omega)$ that is continuous differentiable in $\mathbb{R}^{d}$ with respect to $x$.
    \item [(ii)] Let $DX^{x}(t, \omega) = \frac{\partial X^{x}}{\partial x} (t, \omega)$. Then, $DX^{x}$ is the solution to SDE \eqref{main Fre sdeDX}.
  \end{enumerate}
\end{prop}

\begin{prop} \label{main Fre 2 prop2}
  Let $F : D([0, T], \mathbb{R}^d) \rightarrow D([0, T], \mathbb{R}^{d \times N} )$ be a non-anticipative functional with Assumption~\ref{main F ass}. Assume that $F$ has the second-order Fr\'{e}chet derivative $D^{2}F$, and that $DF$ and $D^{2}F$ are locally Lipschitz continuous in the operator norm, i.e., for each $\eta \in D([0, T], \mathbb{R}^d)$, there exist constants $\epsilon > 0$, $K(\eta)$ such that if $\lVert \eta - \tilde{\eta} \rVert \le \epsilon$ then
  \begin{equation}
  \begin{aligned}\label{main prop2 ass}
    \left\lVert DF(\eta)(\gamma) - DF(\tilde{\eta})(\gamma) \right\rVert      & \le K(\eta) \left\lVert \eta-\tilde{\eta} \right\rVert \left\lVert \gamma \right\rVert ,               \\
    \left\lVert D^{2}F(\eta, \gamma)(\delta) - D^{2}F(\tilde{\eta}, \gamma)(\delta) \right\rVert & \le K(\eta) \left\lVert \eta-\tilde{\eta} \right\rVert \left\lVert \gamma \right\rVert \left\lVert \delta \right\rVert,
  \end{aligned}
  \end{equation}
  for all $\gamma$, $\delta \in D([0, T], \mathbb{R}^d)$. Also, let $D^{2}X^{x}$ be a solution to the SDE for each $x \in \mathbb{R}^{d}$,
  \begin{align}\label{main Fre sdeD2X}
    D^{2}X^{x}_{t} = \int_{0}^{t}D^{2}F_{s-}(X^{x}, DX^{x})(DX^{x}) + DF_{s-}(X^{x})(D^{2}X^{x}) \,dZ_{s}.
  \end{align}
  Then, $D^{2}X^{x}$ exists and is unique. Moreover, we obtain the following.
  \begin{enumerate}
    \item [(i)] For almost everywhere $\omega \in \Omega$, there exists a function $DX^{x}(t, \omega)$ that is continuous differentiable in $\mathbb{R}^{d}$ with respect to $x$.
    \item [(ii)] The differential $\frac{\partial DX^{x}}{\partial x} (t, \omega)$ of $DX^{x}$ is the solution to SDE \eqref{main Fre sdeD2X}.
  \end{enumerate}
\end{prop}

The above theorems state when the initial value of the SDE is a point. 
We can easily extend the theorems to cases in which the initial value of SDE is a path. Let $(\Omega, \mathbb{Q}, \mathcal{F}, (\mathcal{F}_t )_{t \ge 0} )$ be the filtered probability space, where $\mathcal{F}_t$ is a sigma-algebra generated by an $N$-dimensional Wiener process $W_s$ for $0 \le s \le t$. Suppose that the stochastic process $X^{\gamma_{t}}$ follows an SDE for each $t \in [0, T]$ and $\gamma_{t} \in D([0, t], \mathbb{R}^d)$,
\begin{equation} \label{main model sdeX}
  \begin{aligned}
    X^{\gamma_{t}}(s) & = \gamma_{t}(t) + \int^{s}_{t}b_r(X^{\gamma_{t}})\,dr +\int^{s}_{t}\sigma_r(X^{\gamma_{t}})\,dW_r, \quad &&t \le s \le T, \\
    X^{\gamma_{t}}(s) & = \gamma_{t}(s), \quad &&0 \le s \le t, 
  \end{aligned}
  \end{equation}
  where $b$ and $\sigma$ are non-anticipative functionals.
  The proofs of the following theorems are given at the end of Appendix \ref{Appendix:Fre}.

  \begin{thm}\label{main SDE diff}
    Let $b : D([0, T], \mathbb{R}^d) \rightarrow D([0, T], \mathbb{R}^d)$, $\sigma : D([0, T], \mathbb{R}^d) \rightarrow D([0, T], \mathbb{R}^{d \times N} )$ be a non-anticipative functional with Assumption~\ref{main F ass}. Assume that $b, \sigma$ have the Fr\'{e}chet derivatives $Db, D\sigma$ satisfying \eqref{main Fre prop1 ass}. Let $X^{\gamma_{t}}$ be a solution to SDE \eqref{main model sdeX} and $D_{\delta_t}X^{\gamma_{t}}$ be a solution to the SDE 
    \begin{equation}\label{model sdeDX} 
      \begin{aligned}
         & D_{\delta_t}X^{\gamma_{t}}(s) = \delta_{t}(t) + \int_{t}^{s} Db_r(X^{\gamma_{t}})(D_{\delta_{t}}X^{\gamma_{t}}) \,dr + \int_{t}^{s} D\sigma_r(X^{\gamma_{t}})(D_{\delta_{t}}X^{\gamma_{t}}) \,dW_{r} &&\quad t \le s \le T,
        \\ & D_{\delta_t}X^{\gamma_{t}}(s) = \delta_{t}(s) &&\quad 0 \le s < t.
      \end{aligned}  
      \end{equation}
      for each $\delta_t \in D([0, t], \mathbb{R}^d)$. Then, for $(s, \omega) \in [0, T] \times \Omega $ with $\omega$ not in the exceptional set, we have
      \begin{align}
       \lim_{h \rightarrow 0} \left\lvert \frac{X^{\gamma_{t} + h\delta_t }(s, \omega) - X^{\gamma_{t}}(s, \omega)}{h} - D_{\delta_{t}}X^{\gamma_{t}}(s, \omega) \right\rvert = 0,
      \end{align}
      for $h \in \mathbb{R}$. 
  \end{thm}

  \begin{thm}\label{main SDE 2 diff}
    Let $b : D([0, T], \mathbb{R}^d) \rightarrow D([0, T], \mathbb{R}^d)$, $\sigma : D([0, T], \mathbb{R}^d) \rightarrow D([0, T], \mathbb{R}^{d \times N} )$ be a non-anticipative functional with Assumption~\ref{main F ass}. Assume that $b, \sigma$ have Fr\'{e}chet derivatives up to order 2 as $Db, D\sigma, D^2b, D^2\sigma$ satisfying \eqref{main prop2 ass}. Let $X^{\gamma_{t}}$, $DX^{\gamma_{t}}$ be a solution to SDEs \eqref{main model sdeX}, \eqref{model sdeDX} and $D^{2}X^{\gamma_{t}}$ be a solution to the SDE
    \begin{alignat}{3}\label{model sdeD2X}
      D_{\delta'_t}D_{\delta_t}X^{\gamma_{t}}(s) :=& \int_{s}^{t} D^{2}b_r(X^{\gamma_{t}}, D_{\delta_t}X^{\gamma_{t}})(D_{\delta'_t}X^{\gamma_{t}}) + Db_r(X^{\gamma_{t}})(D_{\delta'_t}D_{\delta_t}X^{\gamma_{t}}) \,dr         \\
      +& \int_{s}^{t} D^{2}\sigma_r(X^{\gamma_{t}}, D_{\delta_t}X^{\gamma_{t}})(D_{\delta'_t}X^{\gamma_{t}}) + D\sigma_r(X^{\gamma_{t}})(D_{\delta'_t}D_{\delta_t}X^{\gamma_{t}}) \,dW_{r} \quad &&t \le s \le T, \\
      D_{\delta'_t}D_{\delta_t}X^{\gamma_{t}}(s) =&\;0 \quad &&0 \le s \le t.
  \end{alignat}
  Then, for $(s, \omega) \in [0, T] \times \Omega $ with $\omega$ not in the exceptional set, the equality
\begin{align}
  \lim_{\left\lvert h \right\rvert  \rightarrow 0 } \left\lvert \frac{D_{\delta_t}X^{\gamma_t  + h\delta'_t}(s, \omega) - D_{\delta_t}X^{\gamma_{t}}(s, \omega) - h \cdot D_{\delta'_t}D_{\delta_t}X^{\gamma_{t}}(s) }{h }  \right\rvert = 0
\end{align}
  holds.  
  \end{thm}

The results of the above theorems indicate that the process $X^{\gamma_{t}}$ is differentiable in the direction $\delta_t$ and the derivative is represented as a solution to SDE \eqref{model sdeDX}. We summarize the assumption of coefficients $b, \sigma$ in Theorem~\ref{main SDE 2 diff}.

  \begin{assume} \label{main coeff ass1}
    The non-anticipative functionals $b : D([0, T], \mathbb{R}^d) \rightarrow D([0, T], \mathbb{R}^d)$, $\sigma : D([0, T], \mathbb{R}^d) \rightarrow D([0, T], \mathbb{R}^{d \times N} )$ satisfy the following:
    \begin{enumerate}
      \item [(i)] Assumption~\ref{main F ass},
      \item [(ii)] $b, \sigma$ have Fr\'{e}chet derivatives up to order 2 in sense of Definition~\ref{main Fre def}, and they satisfy locally functional Lipschitz continuity \eqref{main prop2 ass}.
    \end{enumerate}
  \end{assume}
  
\section{Path-dependent option pricing}\label{sec:price}
Through this section, we set up the exponential path-dependent stock price model and apply our results to option pricing. 
Under certain conditions regarding the coefficients and payoff functional, the option price is a $\mathbb{C}^{1, 2}$ solution to a path-dependent partial differential equation (PPDE). 
Using this, we obtain  sensitivity formulas
of option prices for changes of underlying model coefficients.

\subsection{Path-dependent SDEs and PDEs}\label{subsec:PSDE}
We set up the exponential path-dependent stock price model in which the logarithm of stock price follows a general path-dependent SDE. Since this setting is constructed under the risk-neutral model, an option price is represented by a conditional expectation. Next, we introduce a backward stochastic differential equation (BSDE) to which a solution is an option price. Finally, we state a theorem that implies that the option price operator is the solution to a PPDE and a solution to the PPDE is the option price.

Let us denote a logarithm of $d$-dimensional stock price by $X$. Let $r \in \mathbb{R}_+$ be the deterministic short rate term and $q : D([0, T], \mathbb{R}^d) \rightarrow D([0, T], \mathbb{R}^d)$ be the non-anticipative dividend functional. To start under the risk-neutral setting, we define a logarithm of stock price for each $i = 1, \cdots, d$ as
\begin{align} \label{model risk-neutral_1}
  dX^{(i)}_t = (r - q^{(i)}_t(X)- \frac{1}{2}\sigma^{(i)}_t(X)\sigma^{(i)}_t(X)^{\top} )\,dt + \sigma^{(i)}_t(X)\,dW_t, \, X^{(i)}_0 = 0,
\end{align}
where  $W_t$ is $N$-dimensional Brownian motion and $\sigma^{(i)} : D([0, T], \mathbb{R}^d) \rightarrow D([0, T], \mathbb{R}^{N} )$ is a non-anticipative functional. We define a non-anticipative functional $\sigma : D([0, T], \mathbb{R}^d) \rightarrow D([0, T], \mathbb{R}^{d \times N} )$ such that its $i$-th row is $\sigma^{i}$ for all $i = 1, \cdots, d$. Then, the $d$-dimensional process $X$ can be written as 
\begin{align} \label{model risk-neutral}
  dX_t = (r \mathds{1}^{d}-q_t(X)- \frac{1}{2}\left\lvert \sigma_t(X) \right\rvert^{2}  )\,dt + \sigma_t(X)\,dW_t, \, X_0 = \mymathbb{0}^{d},
\end{align}
where $\mymathbb{0}^{d}$, $\mathds{1}^{d}$ are zero and a unit vector, respectively, in $\mathbb{R}^{d}$, and $\lvert \sigma_t \rvert^{2} : D([0, T], \mathbb{R}^d) \rightarrow D([0, T], \mathbb{R}^d)$ is a non-anticipative functional, the $i$-th coordinate of which is $\sigma^{(i)}_t(X)\sigma^{(i)}_t(X)^{\top}$ for $i = 1, \cdots, d$.

The $d$-dimensional stock price is represented by
$  S_{t} = (S^{(1)}_t , \cdots, S^{(d)}_t) := (S^{(1)}_0 e^{X^{(1)}_t}, \cdots, S^{(d)}_0 e^{X^{(d)}_t})$
where $X = (X^{(1)}, \cdots, X^{(d)})$ and $S^{(i)}_0>0$ for all $i = 1, \cdots, d$.
We obtain the equation from the multidimensional It\^{o} formula for all $i = 1, \cdots, d$
\begin{align}
  dS^{(i)}_{t} = S^{(i)}_{t}\,dX^{(i)}_{t} + \frac{1}{2}S^{i}_{t}\,d[X^{(i)}]_{t} = S^{(i)}_{t}(r - q^{(i)}_t(X)) \,dt + S^{(i)}_{t}\sigma^{(i)}_{t}(X)\,dW_{t},  
\end{align}
where $\sigma^{(i)}$ is the $i$-th row of $\sigma_t$.
Thus, the above probability space is a risk-neutral measure space. 

Let $H : D([0, T], \mathbb{R}^d) \rightarrow \mathbb{R}$ be a payoff non-anticipative functional. 
Then, the option price at time $t$ is
\begin{align} \label{model option price}
  e^{-r(T-t)} \mathbb{E}[ H(S_{0}e^{X_{T}}) | \mathcal{F}_{t} ] = e^{-r(T-t)} \mathbb{E}\left[g(X_{T}) | \mathcal{F}_{t} \right],
\end{align}
where $g : D([0, T], \mathbb{R}^d) \rightarrow \mathbb{R}$ with $g^{(i)}(\eta) = H^{(i)} \circ \exp(S_{0} \eta )$ for all $i = 1, \cdots, d$. 

Now, we introduce a BSDE to which a solution is the above option price \eqref{model option price}. Suppose that the stochastic process $X^{\gamma_{t}}$ follows SDE \eqref{main model sdeX}. For a more general setting, the non-anticipative functional $b$ is used instead of $r \mathds{1}^{d}-q_t - \frac{1}{2}\lvert \sigma_t \rvert^{2}$. Note that we set $b = r \mathds{1}^{d}-q_t - \frac{1}{2}\lvert \sigma_t \rvert^{2}$ and $\gamma_{t} = X_{t}$ to ensure that the solution to SDE \eqref{model option price} corresponds to the solution to \eqref{model risk-neutral}. 

In the remainder of this section, we assume that the coefficients $b, \sigma$ satisfy Assumption~\ref{main coeff ass1}. Then, directional derivatives $D_{\delta_t}X^{\gamma_{t}}$ and $D_{\delta'_t}D_{\delta_t}X^{\gamma_{t}}$ of $X^{\gamma_{t}}$ exist for $\delta_t, \delta'_t \in D([0, t], \mathbb{R}^d).$ In particular, we focus on the vertical derivative in Definition~\ref{ftnl ito defi}. Let $e_i$ be the $i$-th standard basis for $\mathbb{R}^{d}$. For a path $\tilde{e}_i := e_i\mathds{1}_{[0,t]} \in D([0, t], \mathbb{R}^d)$, we denote $\partial_iX^{\gamma_{t}}$ as a vertical derivative $D_{\tilde{e}_i}X^{\gamma_{t}}$. In addition, we denote $D_x F(t, \eta)$ as $(\partial_{1}F(t, \eta), \cdots, \partial_{1}F(t, \eta)) \in \mathbb{R}^{1\times d}$.


Let $\mathcal{S}^{2}(0, T ; \mathbb{R})$ be the set of all $(\mathcal{F}_{t})$-adapted processes $Y$ with $\mathbb{E}[\sup_{u \in [0, T]} \lvert Y(u) \rvert^{2}]$ $<\infty$ and $\mathcal{H}^{2}(0, T ; \mathbb{R})$ be the set of all $(\mathcal{F}_{t})$-adapted processes $Z$ with $\mathbb{E}[\int_{0}^{T} \lvert Z(u) \rvert ^{2} \,du] < \infty$. In a similar way, we can define the spaces $\mathcal{S}^{p}$ and $\mathcal{H}^{p}$ for $p \ge 2$. Suppose that $(Y^{\gamma_{t}}, Z^{\gamma_{t}})$ is a solution to BSDE \eqref{model bsdeY}:
\begin{align} \label{model bsdeY}
  Y^{\gamma_{t}}(s) = g(X^{\gamma_{t}}) - \int^{T}_{s} Z^{\gamma_{t}}(r)\,\,dW_r, \quad t \le s \le T,
\end{align}
for $X^{\gamma_{t}}$ in \eqref{main model sdeX}. From the martingale representation theorem, BSDE \eqref{model bsdeY} has a unique solution $(Y^{\gamma_{t}},Z^{\gamma_{t}})_{0 \le t \le T}$ in $\mathcal{S}^{2}(0, T ; \mathbb{R}) \times \mathcal{H}^{2}(0, T ; \mathbb{R})$ for a given $\gamma_{t} \in D([0, t], \mathbb{R}^d)$. Already, we know that $Y^{\gamma_{t}}(u) =\mathbb{E}[ g(X^{\gamma_{t}}_{T}) | \mathcal{F}_{u} ]$ for $ t \le u \le T$.

In addition, we assume that $g : D([0, T], \mathbb{R}^d) \rightarrow \mathbb{R}$ has Fr\'{e}chet derivatives up to order 2.
\begin{assume} \label{model payoff ass2}
  The payoff functional $g : D([0, T], \mathbb{R}^d) \rightarrow \mathbb{R}$ satisfies the following conditions:
  \begin{enumerate}
    \item [(i)] for each $\eta \in D([0, T], \mathbb{R}^d)$, there exists a constant $\delta = \delta(\eta) > 0$ such that if $\lVert \tilde{\eta} - \eta \rVert < \delta $, then $\lVert g(\eta) - g(\tilde{\eta}) \rVert \le C (1 + \lVert \eta \rVert + \lVert \tilde{\eta} \rVert )^{k} \lVert\eta - \tilde{\eta} \rVert$ for some $k = k(\eta) \ge 1$,
    \item  [(ii)] $g$ has Fr\'{e}chet derivatives up to order 2 in sense of Definition~\ref{main Fre def}, and they satisfy local Lipschitz continuity \eqref{main prop2 ass}. More precisely, for each $\eta \in D([0, T], \mathbb{R}^d)$, there exist constant $\epsilon > 0$, $K(\eta)$ such that if $\lVert \eta - \tilde{\eta} \rVert \le \epsilon$ then
    \begin{equation}
    \begin{aligned}\label{main ass g ineq}
      \left\lvert Dg(\eta)(\gamma) - Dg(\tilde{\eta})(\gamma) \right\rvert      & \le K(\eta) \left\lVert \eta-\tilde{\eta} \right\rVert \left\lVert \gamma \right\rVert ,               \\
      \lvert D^{2}g(\eta, \gamma)(\delta) - D^{2}g(\tilde{\eta}, \gamma)(\delta)  \rVert & \le K(\eta) \left\lVert \eta-\tilde{\eta} \right\rvert \left\lVert \gamma \right\rVert \left\lvert \delta \right\rVert,
    \end{aligned}
    \end{equation}
    for all $\gamma$, $\delta \in D([0, T], \mathbb{R}^d)$.
  \end{enumerate}
\end{assume}

Note that we assume the Lipschitz continuity of the coefficients $b$, $\sigma$ and local Lipschitz continuity of the payoff functional $g$. The difference is followed by the existence of solutions to the SDE \eqref{main model sdeX} and BSDE \eqref{model bsdeY}. To guarantee the existence and uniqueness of solutions to the SDE \eqref{main model sdeX} for the entire interval $[0, T]$, we need the coefficients $b$, $\sigma$ to be globally Lipschitz continuous. There may be a blowup in finite time if the coefficients $b$ and $\sigma$ are only locally Lipschitz continuous, not globally. Conversely, we already know that solution to BSDE \eqref{model bsdeY} uniquely exists from the martingale representation theorem. Thus, the first condition of Assumption~\ref{model payoff ass2} is only required for a technical reason.

Now, we introduce the result of this subsection. This result implies that the conditional expectation is the solution to a PPDE. Conversely, it also states that a solution to the PPDE is the conditional expectation. This means that we can estimate the conditional expectation as solving the PPDE. Recall that $\tilde{D}^{d} = \bigcup_{t \in [0, T]} D([0, t], \mathbb{R}^d)$. The proof of this theorem is in Appendix~\ref{proof of main theorem}.

\begin{thm}\label{main theorem}
  Under  Assumptions \ref{main coeff ass1} and \ref{model payoff ass2}, we define a non-anticipative functional  $u : \tilde{D}^{d} \rightarrow \mathbb{R}$ as
  \begin{equation} \label{main def of u}
  u_t(\gamma_{t}) := Y^{\gamma_{t}}(t) = \mathbb{E}\left[ g(X^{\gamma_{t}}_{T}) | \mathcal{F}_{t} \right]
  \end{equation}
  for $\gamma_{t} \in D([0, t], \mathbb{R}^d)$, where $Y^{\gamma_{t}}$ is a solution to BSDE \eqref{model bsdeY}.
  Then  $u$ has first and second-order vertical derivatives and the horizontal derivative, denoted as $D_{x}u,$ $D_{xx}u$, and $D_tu,$ respectively.
 Moreover, $u$ satisfies a PPDE
  \begin{equation}
    \begin{aligned} \label{main PPDE}
    D_{t}u_t(\gamma_{t}) &+ \langle b_t(\gamma_{t}) , D_{x}u_t(\gamma_{t}) \rangle + \frac{1}{2} \textnormal{tr}\left(\sigma_t(\gamma_{t}) D_{xx}u_t(\gamma_{t})\sigma^{\top}_t(\gamma_{t}) \right) = 0, \;\; &&\textnormal{for} \;\; &&\gamma_{t} \in D([0, t], \mathbb{R}^d), \\
    u_T(\gamma_{T})      &= g(\gamma_{T}), \;\; &&\textnormal{for} \;\; &&\gamma_{T} \in D([0, T], \mathbb{R}^d).
  \end{aligned}
  \end{equation}
Conversely, suppose that the regular functional $u$ is a solution to PPDE \eqref{main PPDE}. Then,  
\begin{equation}
u_t(\gamma_{t}) = Y^{\gamma_{t}}(t), \qquad \textnormal{for} \qquad \gamma_{t} \in D([0, t], \mathbb{R}^d),
\end{equation}
where $Y^{\gamma_{t}}$ is a solution to BSDE \eqref{model bsdeY}.
\end{thm}

Now, we set the non-anticipative functional $b_t = r \mathds{1}^{d}-q_t - \frac{1}{2}\lvert \sigma_t \rvert^{2}$ as in \eqref{model risk-neutral} to examine the option price. Then, the PPDE \eqref{main PPDE} is changed to 
\begin{equation}
  \begin{aligned} \label{main PPDE_option}
  D_{t}u_t(\gamma_{t}) &+ \langle r\mathds{1}^{d} - q_t(\gamma_{t}) ,  D_{x}u_t(\gamma_{t})  \rangle  \\
  &+ \frac{1}{2} \textnormal{tr}\left(\sigma_t(\gamma_{t}) \cdot (D_{xx}u_t(\gamma_{t}) - \textnormal{diag}_d(D_xu_t(\gamma_t) )) \sigma^{\top}_{t}(\gamma_{t})  \right) = 0, \, &&\textnormal{for}\,  \gamma_{t} \in D([0, t], \mathbb{R}^d), \\
   u_T(\gamma_{T})      &= g(\gamma_{T})  \, &&\textnormal{for} \, \gamma_{T} \in D([0, T], \mathbb{R}^d),
\end{aligned}
\end{equation}
where $\textnormal{diag}_d(D_xu_t)$ is a $d \times d$ diagonal matrix whose $i$-th coordinate is $\partial_iu_t $ for each $i = 1, \cdots, d$.

Let $u$ be the solution to the above PPDE \eqref{main PPDE_option} and define a non-anticipative functional $v$ as $v_t(\gamma_t) = u_t(\gamma_t)\exp(-r(T-t))$. 
Then, the option price is $v_t(X_t)$, since $v_t(X_t) = u_t(X_t)\exp(-r(T-t)) = \mathbb{E}[g(X_T) | \mathcal{F}_t ]\exp(-r(T-t))$.
Moreover, we can obtain the hedging portfolio of option. This construction is motivated by Theorem~5.1 in \cite{fournie2010functional}. Recall that the $1$-dimensional adapted process $C_{t}$ is a self-financing portfolio with initial value $x$ and $d$-dimensional position process $a_t = (a^{(1)}_t, \cdots, a^{(d)}_t )$ under the $d$-dimensional stock model $S_{t}$ if $C_t$ satisfies 
\begin{equation}
  \begin{aligned} \label{main def portfolio}
    dC_t &=  \frac{C_t -  \langle a_t, S_t \rangle}{R_t} \, dR_t + \sum_{i = 1}^{d} a^{(i)}_t S^{(i)}_t q^{i}_t(X)\, dt + \sum_{i = 1}^{d} a^{(i)}_t  \, dS^{(i)}_t  \\
    &= \Big(  rC_t - \sum_{i = 1}^{d} a^{(i)}_t S^{(i)}_t(r - q^{(i)}_t(X)) \Big) \, dt + \sum_{i = 1}^{d} a^{(i)}_t  \, dS^{(i)}_t \\
        C_0 &= x,
  \end{aligned}
  \end{equation}
where $R_t = \exp(rt)$ denotes bond price with short rate interest $r$ (i.e., $dR_t = rR_t\,dt$).

\begin{prop} \label{greek value prop1}
  Let $u$ be the solution to the above PPDE \eqref{main PPDE_option} and define a non-anticipative functional $v$ as $v_t(\gamma_t) = u_t(\gamma_t)e^{-r(T-t)}$. Then, $v_t(X_t)$ is the self-financing portfolio with the initial value $\mathbb{E}[u_T(X_T)]e^{-rT}$ and the $d$-dimensional position process whose $i$-th coordinate is $\partial_iu_t(X_t)/S^{(i)}_{t}e^{-r(T-t)}$.
\end{prop}

\subsection{Greeks}\label{subsec:G}

For simplicity, we set the dimension $N = d = 1$. We continue using all settings in Section~\ref{subsec:PSDE} except for the dimension.
This section studies the sensitivity of option prices for changes in short rate $r$ and volatility $\sigma$. Since we have $b_t = r - q_t- \sigma^{2}_t/2$ 
to examine option prices, the changes in short rate and volatility can be expressed by the changes in $b$ and $\sigma$, which are the coefficients of SDE \eqref{model risk-neutral}. In this process, we define families of coefficients under certain assumptions. For each coefficient, we compare the value of the portfolio obtained in Proposition~\ref{greek value prop1}. Using this, we obtain the sensitivity formula. This approach is inspired by Chapter~5 in \cite{fournie2010functional}.

First, we 
consider families of 
coefficients  $(r^{\epsilon})_{\epsilon \in I},$ 
$(q_\cdot^{\epsilon})_{\epsilon \in I}$ and $(\sigma_\cdot^{\epsilon})_{\epsilon \in I}$
where $I=(-1,1).$
Here, $\epsilon$ is the perturbation parameter.
For convenience, we define
$ b^{\epsilon}_t :=r^{\epsilon } -q_t^\epsilon- \frac{1}{2}(\sigma^{\epsilon}_t)^{2}.$
For each $\epsilon \in I$, the perturbed stock price is $S^{\epsilon}_{t}= S_{0}e^{X^{\epsilon}} $
where its logarithm $X^{\epsilon}$  is a solution to the SDE
\begin{equation} \label{greek stock model}
\begin{aligned}
 dX^{\epsilon}_{t}& = b^{\epsilon}_t(X^{\epsilon})\,dt + \sigma^{\epsilon}_{t}(X^{\epsilon})\,dW_{t}, \quad X^{\epsilon}_{0} = 0.
\end{aligned}
\end{equation}

\begin{assume} \label{greek assum}
  The family of coefficients $b^{\epsilon}$ and $\sigma^{\epsilon}$ satisfy the following conditions:
  \begin{enumerate}
          \item [(i)] For $\psi = b, \sigma, q$, there exist differentials $\dot{\psi} : [0, T] \times D([0, T], \mathbb{R}) \rightarrow \mathbb{R}$ satisfying the following: for each $\eta_0 \in D([0, T], \mathbb{R})$, there is a $\delta = \delta(\eta_0) > 0$ such that for all $\eta \in D([0, T], \mathbb{R})$ with $ \lVert \eta - \eta_0 \rVert < \delta$,
          \begin{equation} \label{greek diff ineq}
          \begin{aligned}
                \lvert \psi^{\epsilon}_{t}(\eta) - \psi_{t}(\eta) -\epsilon \dot{\psi}_t(\eta) \rvert \le  \epsilon \phi(\eta_0, \epsilon) \left\lVert \eta \right\rVert^{k}_{t},
                \end{aligned}
              \end{equation}      
                where $t \in [0, T]$, $k > 0 $ is independent of $\eta$ and $\phi : D([0, T], \mathbb{R}) \times I \rightarrow \mathbb{R}_{+}$ is continuous with respect to $\epsilon$ with $\phi(\cdot , 0) = 0$.

    \item [(ii)] The families $b^{\epsilon}$, $\sigma^{\epsilon}$ satisfy Assumption~\ref{main coeff ass1}. Moreover, the Lipschitz continuity constant is independent of $\epsilon$. More precisely, for each $\eta \in D([0, T], \mathbb{R}^d)$, there exist constant $\kappa > 0$, $K(\eta) > 0$ such that if $\lVert \eta - \tilde{\eta} \rVert \le \kappa$ then
    \begin{equation}
    \begin{aligned}
      \left\lVert D\phi^{\epsilon}(\eta)(\gamma) - D\phi^{\epsilon}(\tilde{\eta})(\gamma) \right\rVert      & \le K(\eta) \left\lVert \eta-\tilde{\eta} \right\rVert \left\lVert \gamma \right\rVert ,               \\
      \lVert D^{2}\phi^{\epsilon}(\eta, \gamma)(\delta) - D^{2}\phi^{\epsilon}(\tilde{\eta}, \gamma)(\delta)  \rVert & \le K(\eta) \left\lVert \eta-\tilde{\eta} \right\rVert \left\lVert \gamma \right\rVert \left\lVert \delta \right\rVert,
    \end{aligned}
    \end{equation}
    for all $\gamma$, $\delta \in D([0, T], \mathbb{R}^d)$ where $\phi^{\epsilon} = b^{\epsilon}, \sigma^{\epsilon}$.

    \item [(iii)] The differentials   $\dot{b}$, $\dot{\sigma}$  satisfy locally functional Lipschitz continuity, which is the second condition of Assumption~\ref{main F ass}.
    \end{enumerate}
\end{assume}

The first condition indicates the local existence of differentials of the family of coefficients. The second condition means that 
we can apply Theorem~\ref{main theorem} to the stock model in \eqref{greek stock model} for each $\epsilon \in I$. The third condition is needed for a technical reason to prove Theroem~\ref{greek formula thm}.

Let $u$ be a solution to PPDE \eqref{main PPDE_option}. As shown in the paragraph before Proposition~\ref{greek value prop1}, 
the option price at time $0$ is $v_0(X_0) = u_0(X_0)e^{-rT}$ and the perturbed one is $v^{\epsilon}_0(X_{0}) = \mathbb{E}[g(X^{\epsilon}_T)]e^{-r^{\epsilon}T}$ for perturbed short rate term $r^{\epsilon}$. 
From the terminal condition of the PPDE \eqref{main PPDE_option}, we know $g(X^{\epsilon}_T) = u_T(X^{\epsilon}_T)$. Therefore, the sensitivity of option prices is written as 
\begin{align}\label{greek sensitivity form}
  \lim_{\epsilon \rightarrow 0} \frac{1}{\epsilon} (v^{\epsilon}_0(X_{0}) - v_0(X_0) ) =   \lim_{\epsilon \rightarrow 0} \frac{1}{\epsilon} \mathbb{E} [u_T(X^{\epsilon}_T){e^{-r^{\epsilon}T}} - u_T(X_T){e^{-rT}}   ].
\end{align}
To make the term $u_T(X^{\epsilon}_T)$, we apply the portfolio in Proposition~\ref{greek value prop1} to the perturbed stock dynamic $S^{\epsilon}$ in \eqref{greek stock model}. The next proposition states the value of the portfolio at time $t$.
\begin{prop} \label{greek value prop2}
  For each $\epsilon \in I$, define the stochastic process $G^{\epsilon}$ as
  \begin{equation}
  \begin{aligned}
    G^{\epsilon}(t) = u_t(X^{\epsilon}_t) + \int_{0}^{t} D_{x}u_s(X^{\epsilon}_s) \left( b_{s}(X^{\epsilon}) - b^{\epsilon}_{s}(X^{\epsilon}) \right) 
    + \frac{1}{2}D_{xx}u_s(X^{\epsilon}_s) \big( \sigma_{s}^{2}(X^{\epsilon}) - (\sigma^{\epsilon}_s)^{2}(X^{\epsilon}) \big) \,ds,
  \end{aligned}
\end{equation}
  and $K^{\epsilon}(t) = G^{\epsilon}(t)e^{-r^{\epsilon}(T-t)}.$ Then, $G^{\epsilon}$ is a martingale and $K^{\epsilon}$ is a self-financing portfolio under the perturbed stock model $S^{\epsilon}_t $ in \eqref{greek stock model} with initial value $\mathbb{E}[u_T(X^{\epsilon}_T)]e^{-r^{\epsilon}T}$ and the position $\frac{D_{x}u_t(X^{\epsilon})_t}{S^{\epsilon}_{t}}e^{-r^{\epsilon}T}$.
\end{prop}

Using the above arguments, we can obtain the 
following formula, which gives the sensitivity of option price when the short rate $r$ and the volatility $\sigma$ are perturbed.
For the last equality in the theorem, we used   $\dot{r} = \dot{b}_{t} + \dot{\sigma}_{t}\sigma_{t} + \dot{q}_t$ from $ b^{\epsilon}_t :=r^{\epsilon } -q_t^\epsilon- \frac{1}{2}(\sigma^{\epsilon}_t)^{2}$,

\begin{thm} \label{greek formula thm}
    Under Assumptions \ref{main coeff ass1}, \ref{model payoff ass2}, and \ref{greek assum}, the sensitivity of the option price \eqref{greek sensitivity form} is 
    \begin{equation}\label{greek formula}
    \begin{aligned} 
      &\lim_{\epsilon \rightarrow 0} \frac{1}{\epsilon} (v_0^{\epsilon}(X_{0}) - v_0(X_0) ) \\ =\; 
      & - u_0(X_{0})e^{-rT}\dot{r}T 
    - \mathbb{E}\Big[\int_{0}^{T} D_{x}u_s(X_s) \dot{b}_{s}(X) +  D_{xx}u_s(X_s) \dot{\sigma}_{s}(X)\sigma_{s}(X) \,ds\Big]e^{-rT}\\
   =\, &-  u_0(X_{0})e^{-rT}\dot{r}T  
    - \mathbb{E}\Big[\int_{0}^{T} (\dot{r} - \dot{q}_s(X) )D_{x}u_s(X_s)  +  \left(D_{xx}u_s(X_s) -  D_{x}u_s(X_s) \right)\dot{\sigma}_{s}(X)\sigma_{s}(X) \,ds\Big]e^{-rT}.
  \end{aligned}
\end{equation}

\end{thm}

\begin{remark}
  We introduce \cite{fournie2010functional}'s work with respect to the sensitivity analysis using our notation. First, we sketch a model and an option price. The stock model is 
  \begin{align}
    dS_t = r_tS_t\,dt + \sigma_{t}(S)\,dW_t
  \end{align}
  where the deterministic short rate is $r$ and volatility $\sigma$ is a non-anticipative functional satisfying Assumption~\ref{main F ass}. Let $u$ be a locally regular functional in Definition~\ref{ftnl regular def} and $g$ be an option. Assume that for all $t < T$, $u$ satisfies
  \begin{equation}
  \begin{aligned}
    D_{t}u_t(\gamma_{t}) &+ r_tD_xu_t(\gamma_{t})\gamma_{t}(t) + \frac{1}{2}D_{xx}u_t(\gamma_{t})\gamma_{t}^{2}(t)\sigma_t^{2}(\gamma_{t}) = r_tu_t(\gamma_{t}) \\
    u_T(\gamma_{T}) &= g(\gamma_{T})
  \end{aligned}
  \end{equation}
  for $\gamma_{t} \in D([0, t], \mathbb{R}^d)$ and $\gamma_{T} \in D([0, T], \mathbb{R}^d)$. Then, $u_t(S_t)$ is a self-financing portfolio with the initial value $u_0(S_0)$ and position $D_xu_t(S_t)$.

  Now, we turn to sensitivity. Let $(\sigma^{\epsilon})_{\epsilon \in (0, \infty)}$ be families of functionals satisfying the first condition of the Assumption~\ref{main coeff ass1} for each $\epsilon > 0$. 
  We denote the perturbed stock model $S^{\epsilon}$ as the unique solution to 
  \begin{align}
    dS^{\epsilon}_t = r_tS^{\epsilon}_t\,dt + \sigma^{\epsilon}_{t}(S)\,dW_t.
  \end{align}
  Under some additional conditions as an existence of differential $\dot{\sigma}$, 
  we have 
\begin{align}
  \lim_{\epsilon \rightarrow 0} \frac{1}{\epsilon}\mathbb{E}[g(S^{\epsilon}_T) - g(S_T)]e^{-rT}= \mathbb{E} \Big[\int_{0}^{T} \dot{\sigma}_t(S)\sigma_t(S)S^{2}(t)D_{xx}u_t(S_t)e^{-rt}  \Big].
\end{align}
\end{remark}

\section{Example}\label{sec:ex}
\subsection{Coefficients with time integration}\label{subsec:inte}
In this subsection, we consider the no-dividend stock price model with a time integration coefficient as setting $\sigma_t(\eta) = \alpha(\int_{0}^{t}f(t,u)\eta(u)\,du)$ where $\alpha : \mathbb{R} \rightarrow \mathbb{R}$ is a function and $f(t, \cdot) : [0, t] \rightarrow \mathbb{R}$ for each $t \in [0, T]$. The dynamics of the logarithm of stock price follows the SDE 
\begin{align}\label{coeff time model}
  dX_{t} = r -  \frac{1}{2}\alpha^{2}\left( \int_{0}^{t} f(t, u)X_{u}\,du \right)\,dt + \alpha\left( \int_{0}^{t} f(t, u) X_{u} \,du \right)\,dW_{u}.
\end{align}
To ensure that the coefficient $\sigma$ is well-defined, we may assume that for each $t \in [0, T]$, the function $f(t, \cdot) : [0, t] \rightarrow \mathbb{R}$ is in $L^{1}$, i.e., $\lVert f(t, \cdot) \rVert_{L^{1}[0, t]} = \int_{0}^{t} \lvert f(t, u) \rvert \,du < \infty$. Including this condition, there is a sufficient condition for the existence and uniqueness of solution \eqref{coeff time model} and Assumption~\ref{main coeff ass1}.

\begin{prop} \label{coeff prop1}
  Assume that for each $t \in [0, T]$, the norm $\lVert f(t, \cdot) \rVert_{L^{1}[0, t]} = \int_{0}^{t} \lvert f(t, u) \rvert \,du$ is finite. If $\alpha$ and $\alpha^{2}$ are Lipschitz continuous in $\mathbb{R}$, then a solution to the SDE \eqref{coeff time model} uniquely exists. Moreover, if the function $\alpha$ in \eqref{coeff time model} satisfies 
  \begin{enumerate}
    \item [(i)] $\alpha$ is in $C^{2}(\mathbb{R})$, i.e., is continuously differentiable up to order 2,
    \item [(ii)] for each $\eta \in D([0, T], \mathbb{R})$, a path $(  \int_{0}^{t}f(t, u)\eta(u)\,du  )_{t \in [0, T]}$ is c\`{a}dl\`{a}g with respect to $t$,
  \end{enumerate}
  then Assumption~\ref{main coeff ass1} holds.
\end{prop}

\begin{ex}
There are two examples of the function $f$. The first is the Dirac delta function. This means that we define an integral $\int_{0}^{t} f(t, u)\eta(u)\,du = \eta(k(t))$ for some $k(t) \in [0, t]$. Although $f$ is not actually a function, the proof of Proposition~\ref{coeff prop1} holds. Using this, we can understand that the model \eqref{coeff time model} is a generalization of the Markov process.

The other example is to define $f(t, u) = 1/t$ for each $t \in (0, T]$ and define $f(0, 0)$
 as a Dirac delta function at $0$. Then, the path $(  \int_{0}^{t}f(t, u)\eta(u)\,du )_{t \in (0, T]}$ converges to $\eta(0) = \int_{0}^{0}f(0, 0)\eta(u)\,du$ as $t \rightarrow 0$. The second condition of Proposition~\ref{coeff prop1} is easily seen.
Considering these two examples, the last condition of Proposition~\ref{coeff prop1} may not be relaxed.
\end{ex}

\begin{remark}
  The typical example of $\alpha$ satisfying both Lipschitz continuity and differentiability is bounded and $C^{2}(\mathbb{R})$. This includes the sine and cosine functions. If we already know that a solution to the SDE \eqref{coeff time model} exists, we can drop the condition of Lipschitz continuity on $\mathbb{R}$. Moreover, we can reduce a domain of $\alpha$ to $ \{ \int_{0}^{t}f(t, u) X_{u}(\omega)\,du\; | \;t \in [0, T], \omega \in \Omega \}$, which may be a subset of $\mathbb{R}$. 
\end{remark}


\subsection{Greeks}\label{subsec:exG}
In this section, we study the sensitivity of an option price as the short rate $r$ and volatility $\sigma$ change. We follow the notation in Section~\ref{subsec:G}. 
First, we discuss the rho value, which is the partial derivative of the option price with respect to the short rate term $r$. Applying the families of coefficients $b^{\epsilon}_t = b_t + \epsilon$ and $\sigma^{\epsilon}_t = \sigma_t$, we obtain $\dot{\sigma} = 0$ and $\dot{r} = 1$. Hence, the sensitivity of the option price \eqref{greek sensitivity form} is actually the value rho. From Theorem~\ref{greek formula thm}, we obtain 
\begin{align}
  \rho = \frac{\partial v_0(X_0)}{\partial r} = -u_0(X_0)e^{-rT}T - \mathbb{E} \Big[\int_{0}^{T} D_{x}u_{s}(X_s)\,ds \Big]e^{-rT}.
\end{align}

Second, we consider the case in which the volatility changes and the short rate does not. As we study in Section~\ref{subsec:G}, the sensitivity differs from how we define the family of volatility $\sigma^{\epsilon}$. We study two examples in this subsection.

\begin{ex}%
  Define the family of volatility $\sigma^{\epsilon}_t(\eta) = \sigma_t(\eta + \epsilon \mathds{1}_{[0, T]})$. The differential $\dot{\sigma}$ is written as
  \begin{align}
   \dot{\sigma}_t(\eta) =  \alpha'\left( \int_{0}^{t} f(t,u)\eta(u)\, du \right) \int_{0}^{t}f(t,u)\,du,
  \end{align}
  because the inequality \eqref{greek diff ineq} holds from the Fr\'{e}chet derivative of $\sigma$.
   To make the differential of the short rate zero, we define the family of $b^{\epsilon}$ by $b^{\epsilon}_t = r - (\sigma^{\epsilon}_t)^{2}/2$. It is easily seen that the families of coefficients $b^{\epsilon}$, $\sigma^{\epsilon}$ satisfy all conditions of Assumption~\ref{greek assum}. 
  Similarly, we consider the family of the volatility $\sigma^{\epsilon}_t(\eta) = \alpha( \int_{0}^{t} (f(t, u) + \epsilon) \eta(u)\,du )$. Then, differential $\dot{\sigma}$ is written as 
  \begin{align}
    \dot{\sigma}_t(\eta) =  \alpha'\left( \int_{0}^{t} f(t,u)\eta(u) \,du \right) \int_{0}^{t}\eta(u)\,du,
   \end{align}
 and the sensitivity formula can be calculated using the formula \eqref{greek formula}.
\end{ex}

\begin{ex}
  Assume that for each $u \in [0, T]$, there exists a differential $\dot{f}$ of $f$ such that 
  \begin{align}
    \left\lvert f(t+\epsilon , u) - f(t, u) - \epsilon \dot{f}(t, u) \right\rvert  \rightarrow 0 ,
  \end{align}
   as $\epsilon \rightarrow 0$ for all $t \in [u, T]$. We define the family of the volatility  $\sigma^{\epsilon}_{t}(\eta) = \alpha(\int_{0}^{t}f(t + \epsilon,u)(\eta(u)) du)$.
  From the chain rule, we obtain the differential of $\sigma$ as 
  \begin{align}
    \dot{\sigma}_t(\eta)  = \alpha'\left( \int_{0}^{t} f(t,u)\eta(u) \,du \right) \int_{0}^{t}\dot{f}(t,u)\,du.
  \end{align} 
  The sensitivity of an option price is calculated using the formula \eqref{greek formula}.
\end{ex}

\section{Conclusion}\label{sec:conclusion}

We studied the differentiability of solutions to   path-dependent SDEs.
Given an SDE with path-dependent coefficients having  Fr\'{e}chet derivatives, 
for each $\gamma_t, \delta_t \in D([0, t], \mathbb{R}^d)$, 
we proved that the SDE solution $X^{\gamma_{t}}$ is differentiable with respect to the initial path $\gamma_t$ in the direction of $\delta_t$, and that the directional derivative $D_{\delta_t}X^{\gamma_{t}}$ 
is given as a solution to
    \begin{equation} 
\begin{aligned}
& D_{\delta_t}X^{\gamma_{t}}(s) = \delta_{t}(t) + \int_{t}^{s} Db_r(X^{\gamma_{t}})(D_{\delta_{t}}X^{\gamma_{t}}) \,dr + \int_{t}^{s} D\sigma_r(X^{\gamma_{t}})(D_{\delta_{t}}X^{\gamma_{t}}) \,dW_{r} &&\quad t \le s \le T,
\\ & D_{\delta_t}X^{\gamma_{t}}(s) = \delta_{t}(s) &&\quad 0 \le s < t,
\end{aligned}  
\end{equation}
as given in  \eqref{model sdeDX}.

PDE representations and sensitivities of option prices under stock models described in Section~\ref{subsec:PSDE} are studied as applications.
Under Assumptions \ref{main coeff ass1} and \ref{model payoff ass2},
an option price   $u_t(\gamma_{t})=  \mathbb{E}[ g(X^{\gamma_{t}}_{T}) | \mathcal{F}_{t}]$
is differentiable with respect to the time and path, and is given as a solution to PPDE \eqref{main PPDE}:
\begin{equation}
\begin{aligned}
D_{t}u_t(\gamma_{t}) &+ \langle r\mathds{1}^{d} - q_t(\gamma_{t}) ,  D_{x}u_t(\gamma_{t})  \rangle  \\
&+ \frac{1}{2} \textnormal{tr}\left(\sigma_t(\gamma_{t}) \cdot ( (D_{xx}u_t(\gamma_{t}) - \textnormal{diag}_d(D_xu_t(\gamma_t) )) \sigma^{\top}_{t}(\gamma_{t})  \right) = 0  \, &&\textnormal{for }\,  \gamma_{t} \in D([0, t], \mathbb{R}^d)\,, \\
u_T(\gamma_{T})      &= g(\gamma_{T})  \, &&\textnormal{for } \, \gamma_{T} \in D([0, T], \mathbb{R}^d)\,.
\end{aligned}
\end{equation}
Conversely, a regular  solution $u$ of the above PPDE \eqref{main PPDE} satisfies 
\begin{equation}
  u_t(\gamma_{t}) = Y^{\gamma_{t}}(t), \qquad \textnormal{for} \qquad \gamma_{t} \in D([0, t], \mathbb{R}^d),
  \end{equation}
  where $Y^{\gamma_{t}}$ is a solution to BSDE \eqref{model bsdeY}.

Furthermore, we introduced formulas for Greeks of option prices for small changes of the short rate $r$ and volatility $\sigma$. For the families of coefficients $(b^{\epsilon})_{\epsilon \in I}$ and $(\sigma^{\epsilon})_{\epsilon \in I}$, we considered the perturbed stock model $S^{\epsilon}=S_{0}e^{X^{\epsilon}}$ where
\begin{equation} 
  \begin{aligned}
      dX^{\epsilon}_{t}& = b^{\epsilon}_t(X^{\epsilon})\,dt + \sigma^{\epsilon}_{t}(X^{\epsilon})\,dW_{t}, \quad X^{\epsilon}_{0} = 0\,.
  \end{aligned}
  \end{equation}
Then, the perturbed  option price is 
$u_0^\epsilon(X_0)= e^{-r^\epsilon T} \mathbb{E}[g(X_{T}^\epsilon) ]$, and 
under Assumption~\ref{greek assum}
the Greek value is 
expressed as 
\begin{equation}
\begin{aligned} 
    - u_0(X_{0})e^{-rT}\dot{r}T
  - \mathbb{E}\left[\int_{0}^{T} D_{x}u_s(X_s) \dot{b}_{s}(X) +  D_{xx}u_s(X_s) \dot{\sigma}_{s}(X)\sigma_{s}(X) \,ds\right]e^{-rT}.
\end{aligned}
\end{equation} 

As an example, we studied the stock model
having coefficients with time integration form.
The logarithm of the stock price is a solution to SDE \eqref{coeff time model},
\begin{align}
  dX_{t} = r -  \frac{1}{2}\alpha^{2}\left( \int_{0}^{t} f(t, u)X_{u}\,du \right)\,dt + \alpha\left( \int_{0}^{t} f(t, u) X_{u} \,du \right)\,dW_{u},
\end{align}
where $\alpha : \mathbb{R} \rightarrow \mathbb{R}$ is a smooth function and $f(t, \cdot) : [0, t] \rightarrow \mathbb{R}$ is integrable for each $t \in [0, T]$. 
Proposition~\ref{coeff prop1} provided 
conditions on  $\alpha$ and $f$ to satisfy Assumption~\ref{main coeff ass1}. 
We then applied the Greek formulas in \eqref{greek formula} to several types of perturbations. 



\begin{appendices}

\section{Fr\'{e}chet derivative} \label{Appendix:Fre}
In this appendix, we prove Propositions~\ref{main Fre prop1} and \ref{main Fre 2 prop2} as well as Theorems~\ref{main SDE diff} and \ref{main SDE 2 diff}. For simplicity, we only deal with the case of $N = d = 1$. The flow of proofs follows that of Theorem~V.39 in \cite{protter2005stochastic}. First, we define the necessary terms.

Recall that $D([0, T], \mathbb{R})$ is the collection of functions with c\`{a}dl\`{a}g paths from $[0, T]$ to $\mathbb{R}$ and $\mathbb{D}$ is the space of adapted right continuous processes with left limit (RCLL processes for fixed maturity $T \in \mathbb{R}_+$.
For a process $ H \in \mathbb{D} $, we use notations in \eqref{ftnl notation} and define $S^p$ norms in $\mathbb{D}$ for $p \ge 2$
\begin{align}
  \left\lVert H \right\rVert_{S^{p}} =  \Big\lVert \sup_{t \in [0, T]} |H_t| \Big\rVert_{L^p}   = \mathbb{E}[\sup_{t \in [0, T]} \left\lvert H_{t} \right\rvert^{p} ]^{\frac{1}{p}}.
\end{align}
We define the $\mathcal{H}^p$-norm of semimartingale $Z$ with $Z_0=0$ as \'{E}mery norm. See Chapter~5 in \cite{protter2005stochastic} for detail. Note that from Theorem~V.2 in \cite{protter2005stochastic}, $H^{p}$ norm is stronger than $S^{p}$ norm

The following proposition says that the Fr\'{e}chet derivative of non-anticipative functional at $t$ depends only on the path on $[0,t]$. It is understood that the Fr\'{e}chet derivative of a non-anticipative functional is a non-anticipative functional. Recall that we write $\eta^t \in D([0, T], \mathbb{R})$ as the stopped path of $\eta$ at time $t \in [0, T]$.
\begin{prop} \label{Fre prop1}
  Let $F : D([0, T], \mathbb{R})\rightarrow D([0, T], \mathbb{R})$ be a non-anticipative functional with a Fr\'{e}chet derivative and $\eta, \tilde{\eta}, \gamma, \tilde{\gamma} \in  D([0, T], \mathbb{R})$. If $\eta^{t} = \tilde{\eta}^{t}, \gamma^{t} = \tilde{\gamma}^{t}$ for some $t \in [0, T]$, then $DF_{t}(\gamma)(\eta) = DF_{t}(\tilde{\gamma})(\tilde{\eta})$.
\end{prop}
\begin{proof} Let $\delta \in \mathbb{R}_+ $ be the scale of $\eta$. Apply Definition~\ref{main Fre def} to $u = \delta\eta$ and $v = \gamma$. It follows that 
  \begin{equation}
    \begin{aligned}
    0 & = \lim_{\delta \rightarrow 0} \frac{\left\lVert F(\gamma + \delta\eta ) - F(\gamma) - DF(\gamma)(\delta\eta) \right\rVert}{ \left\lVert \delta \eta \right\rVert }
    \ge \lim_{\delta \rightarrow 0} \frac{\left\lvert F_t(\gamma + \delta\eta ) - F_t(\gamma) - DF_t(\gamma)(\delta\eta) \right\rvert}{\left\lVert \delta \eta \right\rVert}                   \\
     & = \lim_{\delta \rightarrow 0} \left\lvert \frac{F_t(\tilde{\gamma} + \delta\tilde{\eta} ) - F_t(\tilde{\gamma})} {\delta \left\lVert \tilde{\eta} \right\rVert }  - DF_t(\gamma)  \left( \frac{\eta}{\left\lVert \tilde{\eta} \right\rVert } \right) \right\rvert \frac{\left\lVert \tilde{\eta} \right\rVert }{ \left\lVert \eta \right\rVert } .
    \end{aligned}
  \end{equation}

  since $DF(\gamma)(\cdot)$ is linear operator in Definition~\ref{main Fre def}. 
  By applying Definition~\ref{main Fre def} to $u = \delta\tilde{\eta}$ and $v = \tilde{\gamma}$, we can obtain what we desire.
\end{proof}

Let $DF$ be the Fr\'{e}chet derivative of non-anticipative funtional $F$. For fixed $t \in [0,T]$ and the adapted processes $X$ and $Y$, the value $DF_t(X)(Y)$ depends only on the paths of $X$ and $Y$ on $[0,t]$. Note that the stopped path $X^t(\cdot) := X(t \wedge \cdot)$ is $\mathcal{F}_t$-measurable. From the above proposition, we have $DF_t(X)(Y) = DF_t(X^t)(Y^t) $ and $DF_t(X^t)(Y^t)$ is $\mathcal{F}_t$-measurable. Thus, as a stochastic process, the operator $DF(X)(Y) : [0, T] \times \Omega \rightarrow \mathbb{R}$ is adapted.

Let us repeat the setup of Proposition~\ref{main Fre prop1}. Let $Z$ be a semimartingale with $Z_{0} = 0$ and $F : D([0, T], \mathbb{R}) \rightarrow D([0, T], \mathbb{R})$ be locally functional Lipschitz continuous with Fr\'{e}chet derivative $DF$. Consider a solution $(X,DX)$ to the SDE with non-anticipative functional coefficient 
\begin{align}
  X^{x}_{t}  & = x + \int_{0}^{t}F_{s-}(X^{x})\,dZ_{s}, \label{Fre SDE_X}           \\
  DX^{x}_{t} & = 1 + \int_{0}^{t}DF_{s-}(X^{x})(DX^{x})\,dZ_{s}. \label{Fre SDE_DX}
\end{align}
Because $F$ and $DF$ are locally functional Lipschitz continuous, we know that the solutions to \eqref{Fre SDE_X} and \eqref{Fre SDE_DX} uniquely exist. 

The following lemma is the restatement of the Lemma~V.2 in \cite{protter2005stochastic}. It says that the solution to SDE \eqref{Fre SDE_X} uniquely exists in $S^{p}$ and its $S^{p}$ norm is estimated independently of the coefficient of the SDE. 

\begin{lemma} \label{Fre Lem1}
  Let $1 \le p \le \infty$, let $J \in S^{p}$, let $F$ be functional Lipschitz continuous satisfying $F(0) = 0$. Then, the equation
  \begin{equation}
    X_{t} = J_{t} + \int_{0}^{t}F_{s-}(X)\,dZ_{s}
  \end{equation}
  has a solution in $S^{p}$. It is unique, and moreover, $\lVert X \rVert _{S^{p}} \le C(p, Z)\lVert J \rVert _{S^{p}}$, where $C(p,Z)$ is a constant depending only on $p$ and $Z$.
\end{lemma}
Note that we can also apply the lemma to the solution to the SDE \eqref{Fre SDE_DX} since the operator $DF(\eta)(\cdot)$ for $\eta \in D([0, T], \mathbb{R})$ is linear.

For the continuity of $X^{x}$ in \eqref{Fre SDE_X}, we apply Theorem~V.37 in \cite{protter2005stochastic}. Using continuity, we can extend Theorem~V.16 in \cite{protter2005stochastic}. Let $\pi = \{0 = T_{0} \le T_{1} \le \cdots T_{k} \le T \}$ denote a finite sequence of finite stopping time. The sequence $\pi$ is called a random partition.
We say a sequence of random partition $\pi_{n} = \{0 = T^{n}_{0} \le T^{n}_{1} \le \cdots T^{n}_{k_{n}} \le T \}$ tends to the identity if
\begin{enumerate} \label{Fre def of tend to identity}
  \item [(i)] $\lim_{n}\sup_{k} T^{n}_{k} = T \; a.s$ and
  \item [(ii)] $\lVert \pi_{n} \rVert := \sup_{k} \lvert T^{n}_{k+1} - T^{n}_{k} \rvert$ converges to 0 $a.s$.
\end{enumerate}

For a stochastic process $Y$, we define approximations of $Y$ as 
\begin{align} \label{Fre def lc part}
  Y^{\pi} \equiv Y_{0} \mathds{1}_{\{0\} } + \sum_{j = 0}^{k} Y_{T_{j}}1_{(T_{j}, T_{j+1}] }, \quad 
  Y^{\pi+} \equiv \sum_{j=0}^{k} Y_{T_{j}}1_{[T_{j}, T_{j+1}) } + Y_T \mathds{1}_{\{T\} }.
\end{align}
If Y is adapted and c\`{a}dl\`{a}g (i.e. $Y \in \mathbb{D}$), then $(Y^{\pi}(s))_{s \ge 0}$ is left continuous with right limits and adapted and $(Y^{\pi+}(s))_{s \ge 0}$ is right continuous. 

The next lemma is an extension of Theorem~V.16 in \cite{protter2005stochastic} to functional Lipschitz continuity.
Using this lemma, we can approximate solutions to the SDEs \eqref{Fre SDE_X} and \eqref{Fre SDE_DX} using solutions to SDEs with function coefficients. 
\begin{lemma} \label{Fre Lem2}
  Suppose that $F$ is functional Lipschitz continuous. Let $X$ be a solution to SDE \eqref{Fre SDE_X} and $X^{(\pi)}$ be a solution to the following SDE:
  \begin{align}
    X^{(\pi)}_{t} = x + \int_{0}^{t} F_s((X^{(\pi)})^{\pi+})^{\pi} \,dZ_s
  \end{align}
  for a random partition $\pi$. If $\pi_{n}$ is a sequence of random partitions tending to the identity, then $X^{(\pi_{n})}$ tends to $X$ in $S^{p}$.
\end{lemma}

\begin{proof} [Idea of proof]
  Note that $X(\cdot, \omega) : [0, T] \rightarrow \mathbb{R}$ is uniformly continuous for $a.e.$ $\omega \in \Omega$. 
  For any small $\epsilon > 0$, there exists large $n$ such that $\lvert X^{\pi_n+}_{u}(\omega) - X_{u}(\omega) \rvert < \epsilon$ for all $u \in [0, t]$.
  We can conclude $\lVert F(X^{\pi+}(\omega)) - F(X(\omega)) \rVert _{t} < C\epsilon$ for some constant $C$ from Lipschitz continuity of $F$. 
  The remaining part is same as in Theorem~V.16 in \cite{protter2005stochastic}.
\end{proof}

 
Now, we prove Proposition~\ref{main Fre prop1}. The construction of the proof follows that of Theorem~V.39 of \cite{protter2005stochastic} with no explosion time.

\begin{proof} [Proof of Proposition~\ref{main Fre prop1}]
  Step 1.
  As in Step 1 of Theorem~V.39 of \cite{protter2005stochastic}, we can assume that $DF$ is globally Lipschitz continuous. This also means that we may assume that the constant $K(\eta)$ in \eqref{main Fre prop1 ass} is independent of the choice $\eta$. Note that from Lipschitz continuity of $DF$, we have the linear growth condition for $DF$, i.e., for all $\eta, \gamma \in D([0, T], \mathbb{R})$, there exist constant $K > 0$ such that 
  \begin{equation}  
  \begin{aligned}\label{Pf lin cond}
    \left\lVert DF(\eta)(\gamma)  \right\rVert  & \le K (1 + \left\lVert \eta \right\rVert ) \left\lVert \gamma \right\rVert.               \\
  \end{aligned}
\end{equation}

  Step 2. We show that $DX^{x}$ is continuous with respect to $x.$
  Let $V(t, \omega) \equiv DX^{x}(t,\omega) - DX^{y}(t, \omega)$ for $t \in [0, T], \omega \in \Omega$. Then, we have
  \begin{equation}
  \begin{aligned}
    V(t) & = \int_{0}^{t}DF_{s-}(X^{x})(V)\,dZ_{s} + \int_{0}^{t}J_{s-}\,dZ_{s} 
  \end{aligned}
  \end{equation}
  where $J_{s} = [DF_{s}(X^{x}) - DF_{s}(X^{y})](DX^{y})$. From the fact that the $H^{p}$ norm is stronger than the $S^{p}$ norm and Emery's inequality, we obtain $    \lVert \int_{0}^{T}J_{s-}\,dZ_{s} \rVert _{S^{p}} \le \lVert J \rVert _{S^{p}} \lVert Z \rVert _{H^{\infty}}.$
  We can estimate $\lVert J \rVert _{S^{p}}$ as  
  \begin{equation} \label{Fre J esti in Sp}
  \begin{aligned}
    \left\lVert J \right\rVert _{S^{p}} & \le K\mathbb{E}\Big[ \sup_{u \in [0, T]} \left\lVert X^{x} - X^{y} \right\rVert^{p}_{u} \left\lVert DX^{y} \right\rVert_{u}^{p}  \Big]^{\frac{1}{p}} \le K\mathbb{E}\left[ \left\lVert X^{x} - X^{y} \right\rVert^{p} \left\lVert DX^{y} \right\rVert^{p}  \right]^{\frac{1}{p}} \\
  & \le CK\mathbb{E} [  \left\lVert X^{x} - X^{y} \right\rVert^{2p}  ]^{\frac{1}{2p}} \left\lVert DX^{y} \right\rVert_{S^{2p}} = CK\left\lVert X^{x} - X^{y} \right\rVert_{S^{2p}} \left\lVert DX^{y} \right\rVert_{S^{2p}}.
  \end{aligned}
\end{equation}
  We use the Lipschitz continuity of $DF$ in the 2nd inequality and H\"{o}lder's inequality in the 4th inequality. The constant may be different for each inequality. Applying Lemma~\ref{Fre Lem1} with $DF(\cdot)(0) = 0$, we obtain
  \begin{equation}\label{Fre DX Lip}
  \begin{aligned}
    \left\lVert V \right\rVert_{S^{p}} & \le C \left\lVert \int_{0}^{T}J_{s-}\,dZ_{s} \right\rVert_{S^{p}}  \le C \left\lVert X^{x} - X^{y} \right\rVert_{S^{2p}} \left\lVert DX^{y} \right\rVert_{S^{2p}}\left\lVert Z \right\rVert_{H^{\infty}}.
  \end{aligned}
\end{equation}

    The Lemma~\ref{Fre Lem1} indicates that $\lVert DX^{y} \rVert_{S^{2p}}$ is bounded in $S^{2p}$. 
    By the proof of Theorem~V.37 in \cite{protter2005stochastic}, we have that $\lVert X^{x} - X^{y} \rVert_{S^{p}} \le C \left\lvert x - y \right\rvert^{p}$, which means that $X^{x}$ is Lipschitz continuous in $S^{p}$ with respect to $x$.
    Now, from \eqref{Fre DX Lip} and Lemma~\ref{Fre Lem1}, we can assert that $DX^{x}$ is continuous with respect to $x$ using Kolmogorov's continuity theorem.

  Step 3. We verify $\frac{\partial X^{x, (n)}}{\partial x} = DX^{x, (n)}$ where $(X^{x, (n)}, DX^{x, (n)})$ is a proper approximation of $(X, DX)$.
  Let $n \in \mathbb{N}$ and $\pi_{n} = \{0 = T^{n}_{0} \le T^{n}_{1} \le \cdots T^{n}_{k_{n}} \le T\}$ be the sequence of random partitions tending to the identity. Define $(X^{(n)}, DX^{(n)})$ is a solution to SDEs for $t \in [0, T]$
  \begin{equation}\label{Fre def tilde}
  \begin{aligned} 
    X^{x, (n)}_{t}  & = x + \int_{0}^{t} F_s((X^{(n)})^{\pi_{n}+})^{\pi_{n}} \,dZ_s,             \\
    DX^{x, (n)}_{t} & = 1 + \int_{0}^{t}DF_{s-}((X^{(n)})^{\pi_{n}+})((DX^{(n)})^{\pi_{n}+})^{\pi_{n}}\,dZ_{s}.
  \end{aligned}
  \end{equation}

 If there is no confusion, we denote $(\tilde{X}^{x}, D\tilde{X}^{x}, T_{k}, \pi )$ or $(\tilde{X}, D\tilde{X}, T_{k}, \pi )$ by $(X^{x, (n)}, DX^{x, (n)}, T^{n}_{k}, \pi_{n})$ for simplicity. First, we prove that $\tilde{X}$ and $D\tilde{X}$ converge to $X$ and $DX$, respectively, in $S^{p}$ space. From Lemma~\ref{Fre Lem2}, we know $\tilde{X} \rightarrow X$ in $S^{p}$. To apply Lemma~\ref{Fre Lem2} to $D\tilde{X}$ in \eqref{Fre def tilde}, an additional discussion is required. For fixed $n$, define operators $G^{n}$, $G$ on $\mathbb{D}$ by 
\begin{equation}\label{Fre SDE G in proof}
\begin{aligned} 
  G^{n}(H) &= DF_{s-}(\tilde{X}^{\pi+})(H^{\pi+})^{\pi}, \qquad G(H) = DF_{s-}(X)(H).
 \end{aligned}
\end{equation} 
for a process $H$.

From Lipscitz continuity \eqref{main Fre prop1 ass}, linear growth \eqref{Pf lin cond} and convergence of $\tilde{X}^x$, we have that $G^{(n)}(DX)$ converges to $G(DX)$ in both supremum norm and $S^p$ norm. 
Then, we obtain $\int_{0}^{t} G^{n}_{s-}(DX)\, dZ_s  \rightarrow \int_{0}^{t} G_{s-}(DX)\, dZ_s$ in $S^{p}$
using the stochastic dominating convergence theorem. as in the proof of Theorem~V.16 in \cite{protter2005stochastic}, we have $D\tilde{X}$ converges to $DX$ in $S^p$ space.

  Now, we prove $\frac{\partial \tilde{X}^{x}}{\partial x}(t, \omega) = D\tilde{X}^{x}(t, \omega)$ for $(t, \omega) \in [0, T] \times \Omega$. We fix $\omega \in \Omega$. 
  From definitions \eqref{Fre def lc part}, \eqref{Fre def tilde}, we have the following equalities for $ t \in [ T_{i}, T_{i+1})$, $ i = 0, \cdots, k_n$, and $s \in [0, t]$,
  \begin{equation}
  \begin{aligned}
    \int_{0}^{t} F_s(\tilde{X}^{\pi+})^{\pi} \,dZ_{s} &= \sum_{j=0}^{i-1}F_{T_{j}}(\tilde{X}^{\pi+}) \cdot  \left( Z_{T_{j+1}} - Z_{T_{j}} \right) + F_{T_{i}}(\tilde{X}^{\pi+}) \cdot  \left( Z_{t} - Z_{T_{i}}\right).
  \end{aligned}
\end{equation}
  Hence, $\tilde{X}$, $D\tilde{X}$ can be described as 
  \begin{equation}\label{Fre dev tilde}
  \begin{aligned} 
    \tilde{X}_{t} & = \tilde{X}_{T_{i}} + F_{T_{i}}(\tilde{X}^{\pi+}) \cdot\left( Z_{t} - Z_{T_{i}}\right)  \\
   D\tilde{X}_{t} &= D\tilde{X}_{T_{i}} +  DF_{T_{i}}(\tilde{X}^{\pi+}) ( (D\tilde{X})^{\pi+}  ) \cdot \left( Z_{t} - Z_{T_{i}}\right)
  \end{aligned}
  \end{equation}
  for $t \in [T_{i}, T_{i+1}]$.

  We use induction to show $\frac{\partial \tilde{X}}{\partial x}(t, \omega) = D\tilde{X}(t, \omega)$.
  If $ t \in [T_{0}, T_{1}] = [0, T_{1}]$, 
  it is easily seen to show $\frac{\partial\tilde{X}_{t}}{\partial x} = D\tilde{X}_{t}$ from Proposition~\ref{Fre prop1}.
  
  Now, we prove the induction step. Suppose that $\frac{\partial \tilde{X}_{t}}{\partial x} = D\tilde{X}_{t}$ for $t \in [0, T_{j}]$. Let $t \in (T_{j}, T_{j+1}]$. From \eqref{Fre dev tilde}, we obtain that $\frac{\partial \tilde{X}^{x}_{t}}{\partial x} = D\tilde{X}^{x}_{T_{j}} + \frac{\partial F_{T_{j}}(\tilde{X}^{x, \pi+})}{\partial x}(Z_{t} - Z_{T_{j}})$.
  We need to show $\frac{\partial F_{T_{j}}(\tilde{X}^{x, \pi+})}{\partial x} =  DF_{T_{j}}( \tilde{X}^{x, \pi+} ) ( D\tilde{X}^{\pi+} )$. For small $\delta > 0$,
  \begin{equation}
    \bigg\lvert \frac{F_{T_{j}}( \tilde{X}^{x+\delta, \pi+} ) - F_{T_{j}}( \tilde{X}^{x, \pi+})}{\delta} - DF_{T_{j}}( \tilde{X}^{x, \pi+} ) ( D\tilde{X}^{\pi+} ) \bigg\rvert  \le AB + C
  \end{equation}
  where $A, B, C$ are defined as 
  \begin{align}
A &= \frac{\lvert F_{T_{j}}( \tilde{X}^{x+\delta, \pi+} ) \! - \! F_{T_{j}}( \tilde{X}^{x, \pi+})\! -\! DF_{T_{j}}( \tilde{X}^{x, \pi+} ) (  \tilde{X}^{x+\delta, \pi+} \!- \!\tilde{X}^{x, \pi+}  )\rvert }{\lVert \tilde{X}^{x+\delta, \pi+} \!- \!\tilde{X}^{x, \pi+} \rVert _{T_{j}}} \\
B &= \frac{\lVert \tilde{X}^{x+\delta, \pi+} \! - \! \tilde{X}^{x, \pi+} \rVert _{T_{j}}}{\delta} \\
C &= \Big\lvert \frac{  DF_{T_{j}}( \tilde{X}^{x, \pi+} ) (  \tilde{X}^{x+\delta, \pi+} - \tilde{X}^{x, \pi+})  }{\delta} - DF_{T_{j}}( \tilde{X}^{x, \pi+} ) ( D\tilde{X}^{\pi+} ) \Big\rvert.
 \end{align}
  Since $\frac{\partial \tilde{X}^{x}_{T_{j}}}{\partial x} = D\tilde{X}_{T_{j}}$, $A$ converges to $0$ as $\delta \rightarrow 0$ from the definition of the Fr\'{e}chet derivative. 
  In addition, since $DF(\eta)(\cdot)$ is a continuous linear operator for each $\eta \in D([0, T, \mathbb{R}])$, 
  we obtain $C \rightarrow 0$
  as $\delta \rightarrow 0$. Now, we derive that  $B$ is bounded uniformly in $\delta$. For $u \in [T_{i}, T_{i+1}]$ with $i = 0, \cdots , j-1$, from \eqref{Fre dev tilde}, we obtain
  $|\tilde{X}^{x + \delta, \pi+}_{u} - \tilde{X}^{x, \pi+}_{u}|  = \lvert F_{T_{i}}(\tilde{X}^{x+\delta, \pi+}) - F_{T_{i}}(\tilde{X}^{x, \pi+}) \rvert \left\lvert Z_{u} - Z_{T_{i}} \right\rvert $.
  For any $\epsilon > 0$,  we can choose $\xi > 0$ such that if $0 < \delta < \xi$ then
  \begin{align}
    \Big\lvert \frac{F_{T_{i}}(\tilde{X}^{x+\delta, \pi+}) - F_{T_{i}}(\tilde{X}^{x, \pi+})}{\delta} - DF_{T_{i}}(\tilde{X}^{x, \pi+})(D\tilde{X}^{\pi+}) \Big\rvert < \epsilon,
  \end{align}
  for all $i = 0, \cdots , j-1$. Thus, we obtain
  \begin{equation}
  \begin{aligned}
    B & = \sup_{u \in [0, T_j]} \frac{\big\lvert \tilde{X}^{x + \delta, \pi+}_{u} - \tilde{X}^{x, \pi+}_{u} \big\rvert }{\delta}  \\
      & \le \left( \big\lVert DF_{T_{i}}(\tilde{X}^{x, \pi+})(D\tilde{X}^{\pi+})\big\rVert_{T_j} +\epsilon \right) \times \sup_{i = 0, \cdots, j-1}\sup_{u \in [T_{i}, T_{i+1}]} \left\lvert Z_{u} - Z_{T_{i}}\right\rvert.
  \end{aligned}
\end{equation}
  We have already fixed $\omega \in \Omega$, and the $\lvert Z_{u} - Z_{T_{i}} \rvert$ part is bounded by $2 \sup_{u\in [0, T_{j}]} \lvert Z_{u} \rvert$. Thus, we have $B$ is uniformly bounded in $\delta$. Consequently, we obtain $\frac{\partial \tilde{X}^{x}_{t}}{\partial x}= D\tilde{X}_{t}$ for all $t \in [0, T]$ from the induction argument.

  Step 4.
  The remaining part is same as Theorem~V.39 of \cite{protter2005stochastic}. From Step 1, we can consider $(X, DX)$ as a distribution called a generalized function. In distribution theory, we use the theorem that if $g$ is a continuous distribution and $Dg$ is a continuous derivative in the distribution sense, then $Dg$ is the derivative of $g$ in classical sense. Therefore, we can conclude $\frac{\partial X}{\partial x}(t, \omega) = DX(t, \omega)$ for $\omega$ not in exceptional set.
\end{proof}

\begin{remark}\label{Fre rmk3}
  The Fr\'{e}chet derivative includes the notion of the usual derivative. If a function $f: \mathbb{R} \rightarrow \mathbb{R}$ is differentiable at $x$, then the Fr\'{e}chet derivative of $f$ at $x \in \mathbb{R}$ is represented by $Df(x) : \mathbb{R} \rightarrow \mathbb{R}$ as $DF(x)(h) = f'(x) \cdot h$. 
  
  The chain rule also holds for the Fr\'{e}chet derivative and we use it for the normed space $D([0, T], \mathbb{R})$ or $\mathbb{R}$. Note that there exists the Fr\'{e}chet derivative of $b(X^{x})$ for an operator $b:D([0, T], \mathbb{R}) \rightarrow D([0, T], \mathbb{R})$ which is represented by $Db(X^{x})(DX^{x})$. 
\end{remark}



We omit the proof of Proposition~\ref{main Fre 2 prop2} since it is the generalize of Proposition~\ref{main Fre prop1} to the second-order Fr\'{e}chet derivative.

\begin{proof}[Proofs of Theorems~\ref{main SDE diff} and \ref{main SDE 2 diff}]
Basically, we follow the proofs of Propositions~\ref{main Fre prop1} and \ref{main Fre 2 prop2}. Here, we only describe differences between the proofs. By switching the initial data from a point to a path, we extend Lemma~\ref{Fre Lem2} to the case of an SDE \eqref{main model sdeX}.
For a fixed initial path $\gamma_{t}$ and derivative direction $\delta_t$, we take a random partition $\pi = \{0 = T_{0} \le T_{1} \le \cdots \le T_{n} = T \} $ including the discontinuity points of $\gamma_{t}$ and $\delta_t$. In the interval of the partition, there is no difference from Lemma~\ref{Fre Lem2}, as the paths $\gamma_{t}$ and $\delta_t$ are continuous in the interval. The discontinuity points at partition $T_0, \cdots, T_n$ can be covered well from the right continuous version $X^{\pi+}$ and $DX^{\pi+}$ because the paths $\gamma_{t}$ and $\delta_t$ are c\`{a}dl\`{a}g. Thus, we have $X^{(\pi_{n})}, DX^{(\pi_n)}$ converges to $X, DX$ in $S^{p}$.
\end{proof}

\section{Non-anticipative PDE}\label{Appendix:PDE}
In this section, we prove Theorem~\ref{main theorem}. The entire proof follows that of \cite{PengWang2016bsde}. 
Applying the traditional BSDE theory to the case of a path-dependent BSDE, we obtain the vertical differentiability of the solution to the BSDE.
For the time derivative, we approximate the given path-dependent BSDE to the case of a classic function coefficient BSDE and use the result \cite{PengWang2016bsde}. In this process, we determine that the non-anticipative functional $u$ in \eqref{main def of u} is the solution to the path-dependent PDE \eqref{main PPDE}.

\subsection{Vertical derivative}\label{Appendix:Ver}
In this part, we show that the operator $u_t(\gamma_{t}) = Y^{\gamma_{t}}(t)$ has a second-order vertical derivative. For this, we show that the $\mathcal{S}^{p}$ norm and $\mathcal{H}^{p}$ norm of solutions to an SDE or BSDE depend only on the initial path, applying the arguments in \cite{Zhang2017BSDE} and \cite{PengWang2016bsde}. After that, we estimate the $S^{p}$ norm of a difference of $Y^{\gamma_{t}}$ when the initial path changes vertically. This argument is deeply involved in \cite{PengWang2016bsde}. Then, we have the differentiability of $u$ using Kolmogorov's continuity theorem.

 Before stating the theorems and lemmas to be proved, we mention an important thing in proofs. As in the proof of Proposition~\ref{main Fre prop1}, we can assume that $b$, $\sigma$ and $g$ are globally Lipschitz continuous. It means that the inequality \eqref{Pf lin cond} holds for $DF = Db, D\sigma$. Moreover, the linear growth for $Dg$ holds, i.e., for all $\eta, \gamma \in D([0, T], \mathbb{R})$, there exists $K > 0 $ such that 
 $\lvert Dg(\eta)(\gamma) \rvert \le K (1 + \left\lVert \eta \right\rVert  ) \left\lVert \gamma \right\rVert$.
 Finally, the constant $C$ in proofs may be different in each line.
 
\begin{thm} \label{Ver thm1}
  For any $p \ge 2$, the following inequality holds,
  \begin{equation}\label{Ver eq in thm1}
    \mathbb{E} \Big[\sup_{s \in [t, T]} |Y^{\gamma_{t}}(s)|^p + \Big| \int_{t}^{T} |Z^{\gamma_{t}}(s)|^2 \,ds\Big|^{\frac{p}{2}} \Big] \le C\mathbb{E}[|g(X^{\gamma_t}) |^p ].
  \end{equation}
\end{thm}
The proof of this theorem is same as Lemma~3.4 in \cite{PengWang2016bsde}. We now study the continuity and differentiability of $(Y^{\gamma_{t}}, Z^{\gamma_{t}})$ with respect to the initial path $\gamma_{t}$. 
 For $\gamma_{t} \in D([0, t], \mathbb{R}^d), \bar{\gamma}_{\bar{t}} \in D([0, \bar{t}], \mathbb{R})$,  we define $\Delta_{h}Y^{\gamma_{t}} = \frac{1}{h}(Y^{\gamma_{t}^{h}} - Y^{\gamma_{t}})$, $\Delta_{h}Z^{\gamma_{t}} = \frac{1}{h}(Z^{\gamma_{t}^{h}} - Z^{\gamma_{t} })$ where $\gamma_{t}^{h} \in D([0, t], \mathbb{R}^d)$ is defined as $\gamma_{t}^{h}(s) = \gamma_{t}(s)\mathds{1}_{[0,t)}(s)  + (\gamma_{t}(t) + h)\mathds{1}_{\{t\}}(s)$ for $s\in [0, t]$.

\begin{lemma} \label{Ver lem1}
  For any $p \ge 2$, there exist some constants $C_p$ and $q$ dependent only on $C, T, k, p$ such that for any $t, \bar{t} \in [0, T]$, $h,\bar{h} \in \mathbb{R} - \{0\}, \gamma_{t} \in D([0, t], \mathbb{R}^d), \bar{\gamma}_{\bar{t}} \in D([0, \bar{t}], \mathbb{R}),$
  \begin{enumerate}
    \item [(i)] $$\mathbb{E} \Big[ \sup_{u\in[t \wedge \bar{t}, T]} |Y^{\gamma_{t}}(u) - Y^{\bar{\gamma}_{\bar{t}}}(u)|^p  \Big]\le C_p  ( 1 + \Vert \gamma_{t} \Vert ^q + \Vert \bar{\gamma}_{\bar{t}} \Vert ^q  ) ( \lVert \gamma_{t} - \bar{\gamma}_{\bar{t}}  \rVert^{p}_{t \wedge \bar{t}} +  \lVert t - \bar{t}  \rVert^{\frac{p}{2}} )  .$$
    \item [(ii)]$$\mathbb{E} \Big[  \Big| \int_{t \wedge \bar{t}}^{T}  ( Z^{\gamma_{t}}(u) - Z^{\bar{\gamma}_{\bar{t}}}(u)  )^2 \,du  \Big|^{\frac{p}{2}}   \Big] \le C_p  ( 1 + \Vert \gamma_{t} \Vert ^q + \Vert \bar{\gamma}_{\bar{t}} \Vert ^q  )  ( \lVert \gamma_{t} - \bar{\gamma}_{\bar{t}}  \rVert^{p}_{t \wedge \bar{t}} +  \lVert t - \bar{t}  \rVert^{\frac{p}{2}} ) . $$
    \item [(iii)]
          \begin{align}
             & \mathbb{E} \Big[ \sup_{u \in [t \wedge \bar{t}, T]} \vert \Delta_{h} Y^{\gamma_{t}}(u) - \Delta_{\bar{h}}Y^{\bar{\gamma}_{\bar{t}}}(u) \vert^p   \Big] \le C_p  ( 1 + \Vert \gamma_{t}^{h} \Vert ^q + \Vert \bar{\gamma}_{\bar{t}}^{\bar{h}} \Vert ^q +   )( \lVert \gamma_{t}^{h} - \bar{\gamma}_{\bar{t}}^{\bar{h}}  \rVert^{p}_{t \wedge \bar{t}} +  \lVert t - \bar{t}  \rVert^{\frac{p}{2}} ).
          \end{align}
    \item [(iv)]\label{iv}
          \begin{align}
             & \mathbb{E} \Big[  \Big| \int_{t \wedge \bar{t}}^{T}  ( \Delta_{h}Z^{\gamma_{t}}(u) - \Delta_{\bar{h}}Z^{\bar{\gamma}_{\bar{t}}}(u)  )^2 \,du  \Big| ^{\frac{p}{2}}   \Big] \le  C_p  ( 1 + \Vert \gamma_{t}^{h} \Vert ^q + \Vert \bar{\gamma}_{\bar{t}}^{\bar{h}} \Vert ^q  )( \lVert \gamma_{t}^{h} - \bar{\gamma}_{\bar{t}}^{\bar{h}}  \rVert^{p}_{t \wedge \bar{t}} +  \lVert t - \bar{t}  \rVert^{\frac{p}{2}} ).
          \end{align}
  \end{enumerate}
\end{lemma}

Before we prove the above lemma, we need to estimate the case of $X^{\gamma_{t}}$. The proof of this lemma is from standard argument in SDE theory.
\begin{lemma} \label{Ver lem2}
  For any $p \ge 2$, there exists some constant $C_p$ dependent only on $C, T, k, p$ such that for any $t, \bar{t} \in [0, T], \gamma_{t} \in D([0, t], \mathbb{R}), \bar{\gamma}_{\bar{t}} \in D([0, \bar{t}], \mathbb{R}),$
  \begin{equation}
    \mathbb{E}\Big[ \sup_{s \in [0, T]}\left\lvert X^{\gamma_{t}}(s) - X^{\bar{\gamma}_{\bar{t}}}(s) \right\rvert^p  \Big] \le C_p \left( 1 + \left\lVert \gamma_{t} \right\rVert^{p}  + \left\lVert \bar{\gamma}_{\bar{t}} \right\rVert ^p \right) (\left\lVert \gamma_{t} - \bar{\gamma}_{\bar{t}} \right\rVert^{p}_{t \wedge \bar{t}} + \left\lvert t - \bar{t} \right\rvert^{\frac{p}{2}} ) .
  \end{equation}
\end{lemma}

Now, we prove Lemma~\ref{Ver lem1}. In this proof, we use the Fr\'{e}chet derivative of $X^{\gamma_{t}}$ introduced in the propositions in Section~\ref{sec:d}. More specifically, let $DX^{\gamma_{t}}$ be the solution to SDE \eqref{main Fre sdeDX}. From Remark~\ref{Fre rmk3}, we obtain the differential of $g(X^{\gamma_{t}^{x}})$, 
\begin{align}\label{Ver diff g}
  \frac{\partial g(X^{\gamma_{t}^{x}}) }{\partial x} = Dg(X^{\gamma_{t}^{x}})(DX^{\gamma_{t}^{x}}).
\end{align}
\begin{proof}[Proof of Lemma~\ref{Ver lem1}]
  The proof follows that of Theorem~3.7 in \cite{PengWang2016bsde}. 
  Without loss of generality, we assume $t \ge \bar{t}$.
  From Theorem~\ref{Ver thm1}, Assumption~\ref{model payoff ass2} and Lemma~\ref{Ver lem2}, we obtain the inequalities 
  \begin{equation}
  \begin{aligned}
    \mathbb{E}\Big[ \sup_{u \in [\bar{t} , T]} \left\lvert Y^{\gamma_{t}}(u) - Y^{\bar{\gamma}_{\bar{t}}}(u) \right\rvert ^{p}  \Big] 
     &\le C\mathbb{E}\Big[ (1 + \left\lVert X^{\gamma_{t}} \right\rVert + \left\lVert X^{\bar{\gamma}_{\bar{t}}} \right\rVert )^{2kp} \Big]^{\frac{1}{2}} \mathbb{E}\Big[\left\lVert  X^{\gamma_{t}} - X^{\bar{\gamma}_{\bar{t}}}  \right\rVert^{2p}  \Big]^{\frac{1}{2}} \\
     &\le C \left( \left\lVert \gamma_{t} - \bar{\gamma}_{\bar{t}} \right\rVert^{p}_{t \wedge \bar{t}} + \left\lvert t - \bar{t} \right\rvert^{\frac{p}{2}} \right)  \left( 1 + \left\lVert \gamma_{t} \right\rVert  + \left\lVert \bar{\gamma}_{\bar{t}} \right\rVert \right)^{kp}.
  \end{aligned}
  \end{equation}
  The proof of $(ii)$ in the lemma is same as in Theorem~\ref{Ver thm1} regarding the control of Z.
  
  To show $(iii)$, let us consider $(\Delta_{h}Y^{\gamma_{t}} - \Delta_{\bar{h}}Y^{\bar{\gamma}_{\bar{t}}}, \Delta_{h}Z^{\gamma_{t}} - \Delta_{\bar{h}}Z^{\bar{\gamma}_{\bar{t}}}) $ as a solution to the BSDE
  \begin{equation}
    Y(r) = \frac{g(X^{\gamma^{h}_{t}}) - g\left( X^{\gamma_{t}} \right)}{h} -\frac{ g(X^{\bar{\gamma}^{\bar{h}}_{\bar{t}}}) - g\left( X^{\bar{\gamma}_{\bar{t}}} \right)}{\bar{h}}  - \int^{T}_{r} Z(r)\,dW_r, \quad \bar{t} \le r \le T.
  \end{equation}
  
  Applying \eqref{Ver diff g}, we estimate $(iii)$ in the lemma as above. We obtain that
  \begin{equation}
    \begin{aligned}\label{ver estimate in lem2}
       & \mathbb{E}\Big[ \sup_{s \in [\bar{t}, T]}\Big\lvert \Delta_{h}Y^{\gamma_{t}}(s) - \Delta_{\bar{h}}Y^{\bar{\gamma}_{\bar{t}}}(s) \Big\rvert  ^p  \Big] \le C \mathbb{E}\Big[ \Big\lvert \frac{g(X^{\gamma^{h}_{t}}) - g\left( X^{\gamma_{t}} \right)}{h} -\frac{ g(X^{\bar{\gamma}^{\bar{h}}_{\bar{t}}}) - g\left( X^{\bar{\gamma}_{\bar{t}}} \right)}{\bar{h}}  \Big\rvert ^p  \Big]\\
        =  \; &C\mathbb{E}\Big[ \Big\lvert \int_{0}^{1} Dg(X^{\gamma^{\theta h}_{t}})(DX^{\gamma^{\theta h}_t}) - Dg( X^{\bar{\gamma}^{\theta\bar{h}}_{\bar{t}}} )( DX^{\bar{\gamma}^{\theta\bar{h}}_{\bar{t}}} ) \,d\theta \Big\rvert^{p} \Big]                          
    \end{aligned}
  \end{equation}
  We use the fundamental theorem of calculus to estimate the second equality. These terms can be estimated using Lipschitz continuity of $Dg$.

The proof of the last term is similar to the proof of Theorem~\ref{Ver eq in thm1}.
\end{proof}

The next step is to apply Kolmogorov's continuity theorem. Note that this theorem is the version of Lemma~3.8 in \cite{PengWang2016bsde}.
\begin{thm} \label{Ver thm2}
  For each $\gamma_{t} \in D([0, t], \mathbb{R})$, $\{Y^{\gamma_{t}^{x} }(s) : s \in [0, T], x \in \mathbb{R}\}$ has a version which is $a.e$ in $C^{0, 2}([0, T] \times \mathbb{R})$.
\end{thm}
\begin{proof}
  
  Fix $k, \bar{k} \in \mathbb{R}$. From the second inequality in Lemma~\ref{Ver lem1}, we obtain
that $(Y^{\gamma_{t}^{x}}, Z^{\gamma_{t}^{x}})$ is continuous with respect to $x$ in $\mathcal{S}^{2}$ and $\mathcal{M}^{2}$ from Kolmogorov's continuity theorem. Let $(D_{x}Y^{\gamma_{t}}, D_{x}Z^{\gamma_{t}}) $ be a solution to the BSDE
  \begin{align} \label{Ver dyn of DY in thm2}
    D_{x}Y^{\gamma_{t}}(s) = Dg(X^{\gamma_{t}})(DX^{\gamma_{t}}) - \int_{s}^{T}D_{x}Z^{\gamma_{t}}(u)\,du.
  \end{align}

    As in the proof of Lemma~\ref{Ver lem1}, we have 
    \begin{equation}
      \begin{aligned}
        \mathbb{E}\Big[ \big\lvert \frac{ g(X^{\gamma^{h}_{t}}) - g(X^{\gamma_{t}})}{h} - Dg(X^{\gamma_{t}})(DX^{\gamma_{t}}) \big\rvert ^{p} \Big] \le C (1 +\lVert \gamma_{t}^{h} \rVert^{q} )\lVert \gamma_{t}^{h} - \gamma_{t} \rVert^{p}.
      \end{aligned}
      \end{equation}
      It also means that $\Delta_{h}Y^{\gamma_{t}}$ converges to $D_xY^{\gamma_t}$ as $h \rightarrow 0$ in $S^p$ space.
  Using the $(ii)$ in Lemma~\ref{Ver lem1}, it is easy to show the continuity of $D_xY^{\gamma_{t}^{x}}$ with respect to $x$.

  Similarly, let $(D_{xx}Y^{\gamma_{t}}, D_{xx}Z^{\gamma_{t}})$ be a solution to the BSDE
  \begin{equation}
    D_{xx}Y^{\gamma_{t}}(s) = D^{2}g(X^{\gamma_{t}}, DX^{\gamma_{t}})(DX^{\gamma_{t}}) + Dg(X^{\gamma_{t}})(D^{2}X^{\gamma_{t}}) - \int_{s}^{T} D_{xx}Z^{\gamma_{t}}_{s}\,dW_{u},
  \end{equation}
  where $D^{2}X^{\gamma_{t}} = \frac{\partial^{2} DX^{\gamma^{h}_{t}}}{\partial h} \big|_{h = 0 }$. The same proof woks for $D_{xx}Y^{\gamma_{t}}$.
\end{proof}

\subsection{Time derivative}\label{Appendix:Time}
In this subsection, we prove the horizontal differentiability  of $u$ in \eqref{main def of u}. In addition, the operator $u$ is the solution to the path-dependent PDE. The entire idea follows the flow of Theorems~3.10 and 4.5 in \cite{PengWang2016bsde}. Recall that the operator $u : \tilde{D} \rightarrow \mathbb{R}$ is defined as $u_t(\gamma_{t}) = Y^{\gamma_{t}}(t)$, where $\tilde{D}= \bigcup_{t \in [0, T]} D([0, t], \mathbb{R})$.

\begin{thm} \label{Time thm1}
  The operator $u$ in \eqref{main def of u} has the horizontal derivative $D_t u : \tilde{D} \rightarrow \mathbb{R}$ and $u$ satisfies a path-dependent PDE:
  \begin{equation}
    \begin{aligned} \label{Tim ppde in thm1}
    D_{t}u_t(\gamma_{t}) &+ b_t(\gamma_{t}) D_{x}u_t(\gamma_{t}) + \frac{1}{2}\sigma_t^{2}(\gamma_{t}) D_{xx}u_t(\gamma_{t}) = 0, \quad &&\textnormal{for} \quad &&\gamma_{t} \in D([0, t], \mathbb{R}), \\
    u_T(\gamma_{T})      &= g(\gamma_{T}), \quad &&\textnormal{for} \quad &&\gamma_{T} \in D([0, T], \mathbb{R}).
  \end{aligned}
  \end{equation}
\end{thm}
  We denote a differential operator $Lu =  b D_{x}u + \frac{1}{2}\sigma^{2} D_{xx}u$. First, we give the main ideas of the proof. For the initial path $\gamma_{t} \in D([0, t], \mathbb{R})$, define the horizontally extended path $\gamma_{t, t+\delta} \in D([0, t+\delta], \mathbb{R})$ as $\gamma_{t+\delta}(s) = \gamma_t(s) \mathds{1}_{[0, t]} + \gamma_{t}(t) \mathds{1}_{(t, t+\delta]}$ for small $\delta > 0$. The main idea is to prove each of the following equalities
\begin{equation}\label{Tim main idea}
  \begin{aligned}
  &u_{t + \delta}(\gamma_{t, t + \delta}) - u_{t}(\gamma_{t}) = \lim_{n \rightarrow \infty} u^{(n)}(\gamma_{t, t+\delta}) - u^{(n)}(\gamma_{t}) \\
  = \; & \lim_{n \rightarrow \infty}\int_{t}^{t+\delta} \partial_{t}\tilde{u}^{(n)}(s, \gamma_{t}(t_0), \cdots, \gamma_{t}(t_j))\,ds \\
   =  \;  &\lim_{n \rightarrow \infty} \int_{t}^{t+\delta} -\tilde{L}^{(n)}(s, \gamma_{t}(t_0), \cdots, \gamma_{t}(t_{j(n)}))\tilde{u}^{(n)}(s, \gamma_{t}(t_0), \cdots, \gamma_{t}(t_{j(n)}))\,ds \\
  = \; &- \int_{t}^{t+\delta}L(\gamma_{t, s})u_{s}(\gamma_{t, v})\,ds,
\end{aligned}
\end{equation}
for all small $\delta > 0$ and some suitable $u^{(n)}$, $\tilde{u}^{(n)}$ and $\tilde{L}^{n}$.
After that, we obtain that 
\begin{align}\label{Time aim}
 D_tu_t(\gamma_{t}) := \lim_{\delta \rightarrow 0} \frac{u_{t + \delta}(\gamma_{t, t+\delta}) - u_{t}(\gamma_{t}) }{\delta} = -Lu_{t}(\gamma_{t}).
\end{align}
 This means that $u$ has a horizontal derivative, and also that $u$ is a solution to the path-dependent PDE
\begin{align}
  D_{t}u_t(\gamma_t) + Lu_t(\gamma_t) = 0, \quad \textnormal{for} \quad  \gamma_{t}\in D([0, t], \mathbb{R}).
\end{align}

First, we introduce $X^{n, \gamma_{t}}$ and $(Y^{n, \gamma_{t}}, Z^{n, \gamma_{t}})$ by approximating the coefficient function $b$, $\sigma$ and define the operator $u^{(n)}$ in \eqref{Tim main idea}. The approximation may be considered as the detailed version in Proposition~\ref{main Fre prop1}.

Fix the time $t \in [0, T]$ and initial path $\gamma_{t} \in D([0, t], \mathbb{R})$. Let $n \in \mathbb{N}$ be a number sufficiently larger than the number of discontinuity points of $\gamma_{t}$. Then, we can define a sequence of partitions $\pi_n = \{ 0 = t_0 \le t_1 \le \cdots \le t_n = T \}$ tending to the identity (see the paragraph after Lemma~\ref{Fre Lem1}) which covers all discontinuity points of $\gamma_{t}$ and the time $t$. We also define linear operators  $\varphi^{n}_{L}, \varphi^{n}_R : \tilde{D} \rightarrow \tilde{D}$ as for $s \in [0, T]$:
\begin{equation}
\begin{aligned}
  (\varphi^{n}_{L})_v(\eta) &=  \sum_{j=0}^{n-1} \eta(t_{j} \wedge s) \mathds{1}_{[t_{j} \wedge s , t_{j+1}\wedge s )}(v) + \eta(T\wedge s)\mathds{1}_{\{T \wedge s \} }(v), \\
  (\varphi^{n}_{R})_v(\eta) &=  \sum_{j=0}^{n-1} \eta(t_{j+1} \wedge s) \mathds{1}_{[t_{j} \wedge s , t_{j+1}\wedge s )}(v) + \eta(T\wedge s)\mathds{1}_{\{T \wedge s \} }(v),
\end{aligned}
\end{equation}
for $\eta \in D([0, s], \mathbb{R})$ and $v \in [0, s]$ where $\tilde{D} = \bigcup_{t \in [0, T]}D([0, t], \mathbb{R})$ defined in Section~\ref{ftnl Ito formula}. The operators $\varphi^{n}_L$ and $\varphi^{n}_R$ convert the path to a piecewise constant path. They help us to approximate a non-anticipative functional by some good function.

For all sufficiently large $n \in \mathbb{N}$ and a path $\bar{\gamma}_{\bar{t}} \in D([0, \bar{t}], \mathbb{R})$, we define the process $X^{n, \bar{\gamma}_{\bar{t}}}$ as a solution to the SDE
\begin{equation}
  \begin{aligned} \label{Tim dyn of Xn}
    X^{n, \bar{\gamma}_{\bar{t}}}(s) &= \bar{\gamma}_{\bar{t}}(s) + \int_{t}^{s} (b \circ \varphi^{n}_{L})_r(X^{n, \bar{\gamma}_{\bar{t}}} )\,dr + \int_{t}^{s} (\sigma \circ \varphi^{n}_{L})_r (X^{n, \bar{\gamma}_{\bar{t}}})\,dW_{r}, \quad &&s \in [\bar{t} , T], \\
    X^{n, \bar{\gamma}_{\bar{t}}}(s) &= \bar{\gamma}_{\bar{t}}(s), \quad &&s \in [0, \bar{t}].
  \end{aligned}
  \end{equation}
Similarly, we define $(Y^{n, \bar{\gamma}_{\bar{t}}} , Z^{n, \bar{\gamma}_{\bar{t}}})$ as a solution to the BSDE
\begin{align}
Y^{n, \bar{\gamma}_{\bar{t}}}(s) = g \circ \varphi^{n}_{R}(X^{n, \bar{\gamma}_{\bar{t}}}) - \int_{s}^{T} Z^{n, \bar{\gamma}_{\bar{t}}} (r) \,dW_{r}.
\end{align}
Then, we can consider the derivative $DX^{n, \bar{\gamma}_{\bar{t}}}$ and $D_xY^{n, \bar{\gamma}_{\bar{t}}}$ as in \eqref{model sdeDX} and \eqref{Ver dyn of DY in thm2}. We set $u^{(n)} (\bar{\gamma}_{\bar{t}}) = Y^{n, \bar{\gamma}_{\bar{t}}}(\bar{t})$ and $D_xu^{(n)}(\bar{\gamma}_{\bar{t}})  = D_xY^{n, \bar{\gamma}_{\bar{t}}}(\bar{t})$. Then, $u^{(n)}, D_xu^{(n)}  : \tilde{D} \rightarrow \mathbb{R}$ is the approximation of $u, D_xu$.

\begin{lemma} \label{Time u^n Dxu^n}
     Fix $\bar{\gamma}_{\bar{t}} \in D([0, \bar{t}], \mathbb{R})$. If the discontinuity points of $\bar{\gamma}_{\bar{t}}$ are included in those of $\gamma_{t}$, we have that $u^{(n)}(\bar{\gamma}_{\bar{t}})$ and $D_xu^{(n)}(\bar{\gamma}_{\bar{t}})$ converge to $u(\bar{\gamma}_{\bar{t}})$ and $ D_x(\bar{\gamma}_{\bar{t}})$ as the mesh $n \rightarrow \infty$.
\end{lemma}
\begin{proof}
    First, we need to show $X^{n, \bar{\gamma}_{\bar{t}}}$ and $DX^{n, \bar{\gamma}_{\bar{t}}}$ converge to $X^{\bar{\gamma}_{\bar{t}}}$ and $DX^{\bar{\gamma}_{\bar{t}}}$. From the simple extension of Theorem~3.2.2 and Theorem~3.2.4 in \cite{Zhang2017BSDE} to the case of path-dependent, it is enough to show 
    \begin{equation} \label{Time lim_1}
    \mathbb{E}\Big[ \Big\lvert   \int_{t}^{T} F^{(1)}_r(X^{\bar{\gamma}_{\bar{t}}})(DX^{\bar{\gamma}_{\bar{t}}}) - F^{(2)}_r(X^{\bar{\gamma}_{\bar{t}}})(DX^{\bar{\gamma}_{\bar{t}}}) \,dr  \Big\rvert^{p} \Big] \rightarrow 0,
  \end{equation}
where $F^{(1)} = b, \sigma, Db, D\sigma$ and $F^{(2)} = b\circ \varphi^{n}_L, \sigma \circ \varphi^{n}_L, D[b \circ \varphi^{n}_L], D[\sigma \circ \varphi^{n}_L]$.

  Since $\varphi^{n}_{L}$ is a linear operator and $\frac{\lVert \varphi^{n}_{L}(\eta) \rVert }{\lVert \eta \rVert } = \lVert \varphi^{n}_{L}(\frac{ \eta }{\lVert \eta \rVert}) \rVert \le 1$ from the definition, we obtain the Fr\'{e}chet derivative of $F \circ \varphi$ as $D[F \circ \varphi^{n}_{L}](\eta)(\delta) = DF(\varphi^{n}_{L}(\eta))(\varphi^{n}_{L}(\delta))$.
Using Lipschitz continuity, linear growth and the dominant convergence theorem, we have $X^{n, \bar{\gamma}_{\bar{t}}}, DX^{n, \bar{\gamma}_{\bar{t}}} \rightarrow X^{\bar{\gamma}_{\bar{t}}}, DX^{\bar{\gamma}_{\bar{t}}}$ in $S^p$ space.

    Since $u^{(n)}(\bar{\gamma}_{\bar{t}})$, $D_xu^{(n)}(\bar{\gamma}_{\bar{t}})$, $u_{\bar{t}}(\bar{\gamma}_{\bar{t}})$, $ D_xu_{\bar{t}}(\bar{\gamma}_{\bar{t}})$ are all solution to some BSDE, we can use the argument in Lemma~\ref{Ver eq in thm1}. Thus, showing that 
    \begin{equation}\label{Time eq:12}
  \mathbb{E}[ \lvert D[g \circ \varphi^{n}_{R}](X^{n, \bar{\gamma}_{\bar{t}}}) ( DX^{n, \bar{\gamma}_{\bar{t}}}) - Dg(X^{\bar{\gamma}_{\bar{t}}})(DX^{\bar{\gamma}_{\bar{t}}}) \rvert^{p}  ],
   \mathbb{E}[ \lvert g(X^{\bar{\gamma}_{\bar{t}}} ) - g \circ \varphi^{n}_{R}(X^{n, \bar{\gamma}_{\bar{t}}}) \rvert^p ]
  \end{equation}
  converges to $0$ as $n \rightarrow \infty$ completes the proof. Note that this is easy to check from Lipschitz continuity and linear growth of $g, Dg$.
\end{proof}

The above Lemma~\ref{Time u^n Dxu^n} can be extended to the case of $D_{xx}u^{(n)}( \bar{\gamma}_{\bar{t}} ) \rightarrow D_{xx}u( \bar{\gamma}_{\bar{t}} )$.
The next step is to find the family of functions $(u_{k})_{k = 0, 1, \cdots, n}$ satisfying
\begin{equation}\label{Time 2nd step equality}
  Y^{n, \bar{\gamma}_{\bar{t}}}(\bar{t}) = u^{(n)}(\bar{\gamma}_{\bar{t}}) = u_{k}(\bar{t},\bar{\gamma}_{\bar{t}}(t_{0}),\cdots , \bar{\gamma}_{\bar{t}}(t_{k}),\bar{\gamma}_{\bar{t}}(\bar{t}))
\end{equation}
for $\bar{\gamma}_{\bar{t}} \in D([0, \bar{t}], \mathbb{R})$ with $\bar{t} \in [t_{k}, t_{k+1}]$. In this process, we determine that $u_k$ is actually a solution to some BSDE and is a smooth solution to some PDE with a differential operator. Later, we approximate the differential operator to what we desire.

To do this, we need to regard $X^{n, \bar{\gamma}_{\bar{t}}}$ as a solution to a function coefficient SDE. First, we fix $s \in [t_k, t_{k+1})$ for some $k = 0, \dots, n-1$.
For each path $\eta \in D([0, T], \mathbb{R})$, the value $b_{s}(\varphi^{n}_L(\eta))$ depends only on the k points of a path $\eta$, since the operator $b$ is non-anticipative functional. Thus, we can define $b^{(n)}_{k} : [t_{k}, t_{k+1}) \times \mathbb{R}^{k+1} \rightarrow \mathbb{R}$ for each $k = 0, \cdots, n-1$, 
\begin{equation} \label{Time def of bnk}
  b^{(n)}_{k}(s, \eta(t_{0}) , \eta(t_{1}), \cdots , \eta(t_{k}))  = (b \circ \varphi^{n}_L)_s(\eta)
\end{equation}

Similarly, we define $\sigma^{(n)}_{k}$ for each $k = 0, \cdots, n-1$. Then, the dynamics of $X^{n, \bar{\gamma}_{\bar{t}}}$ is represented using the notation $b^{(n)}_{k}, \sigma^{(n)}_{k}$ on each time interval $s \in [t_{k}, t_{k+1})$ by
\begin{equation} \label{Time SDE version}
\begin{aligned}
  X^{n, \bar{\gamma}_{\bar{t}}}(s) = X^{n, \bar{\gamma}_{\bar{t}}}(t_k) &+ \int_{t_k}^{s} b^{(n)}_k\left(r, X^{n, \bar{\gamma}_{\bar{t}}}(t_0), \cdots, X^{n, \bar{\gamma}_{\bar{t}}}(t_k) \right)\,dr \\
  & +\int_{t_k}^{s} \sigma^{(n)}_k\left(r, X^{n, \bar{\gamma}_{\bar{t}}}(t_0), \cdots, X^{n, \bar{\gamma}_{\bar{t}}}(t_k) \right)\,dW_r.
\end{aligned}
\end{equation}
In the same way, we define $g^{(n)}$ by the function representation of $g \circ \varphi^{n}_{R}$ as 
\begin{equation}
  g^{(n)}(\eta(t_0), \cdots, \eta(t_n))  = g \circ \varphi^{n}_R(\eta).
\end{equation}
 Then, the dynamics of $(Y^{n, \bar{\gamma}_{\bar{t}}}, Z^{n, \bar{\gamma}_{\bar{t}}})$ satisfies the following for each time interval $s \in [t_{k}, t_{k+1})$:
\begin{equation} \label{Time BSDE version}
\begin{aligned}
  Y^{n, \bar{\gamma}_{\bar{t}}}(s) & = Y^{n, \bar{\gamma}_{\bar{t}}}(t_{k+1}) - \int_{s}^{t_{k+1}}Z^{n, \bar{\gamma}_{\bar{t}}}(r)\,dW_{r}, \quad \textnormal{for} \quad s \in [t_{k}, t_{k+1}), \\
  Y^{n, \bar{\gamma}_{\bar{t}}}(T) & = g^{(n)}(X^{n, \bar{\gamma}_{\bar{t}}}(t_{0}), X^{n, \bar{\gamma}_{\bar{t}}}(t_{1}),  \cdots , X^{n, \bar{\gamma}_{\bar{t}}}(T) ).
\end{aligned}
\end{equation}
Finally, we define the differential operator $L^{(n)}_{k} = b^{(n)}_{k} \frac{\partial}{\partial x_k} + \frac{1}{2}(\sigma^{(n)}_{k})^{2} \frac{\partial}{\partial x_k}$, where $x_k$ is the last coordinate for each $k = 0, \cdots, n-1$.

 Now, we find a PDE from the last interval which the solution to BSDE $Y^{n, \bar{\gamma}_{\bar{t}}}(\bar{t})$ satisfies. 
Let consider for $[t_{n-1}, t_n]$ in \eqref{Time SDE version} and \eqref{Time BSDE version}.
If $X^{n, \bar{\gamma}_{\bar{t}}}(t_0)$, $\cdots$, $X^{n, \bar{\gamma}_{\bar{t}}}(t_{n-1})$ are known, we can regard \eqref{Time SDE version} and \eqref{Time BSDE version} as standard SDE and BSDE with function coefficients. Thus, we can apply the result in \cite{pardoux1992backward}.
For parameters $x_0, \cdots, x_{n-1} \in \mathbb{R}$, let $u_{n-1} : [t_{n-1}, t_n] \times \mathbb{R}^{n} \times \mathbb{R} \rightarrow \mathbb{R}$ be a solution to the PDE
\begin{equation}\label{Tim pde with index n-1}
\begin{aligned} 
  \partial_{t}u_{n-1}(s, x_{0}, \cdots, x_{n-1}, y) &+ L^{(n)}_{n-1}(s, x_{0}, \cdots, x_{n-1})u_{n-1}(s, x_{0}, \cdots, x_{n-1}, y) = 0, \\
  u_{n-1}(T, x_{0}, \cdots, x_{n-1}, y) &= g^{(n)}(x_{0}, \cdots, x_{n-1}, y),
\end{aligned}
\end{equation}
 for $s \in [t_{n-1}, t_n]$ and $y \in \mathbb{R}$. The parameter determine coefficients of PDE and the solution at $(t_{n-1}, y)$ means that the initial point of SDE is $y$ at time $t_{n-1}$. Thus, we have from Theorem~3.2 in \cite{pardoux1992backward}
\begin{equation}\label{Time terminal condition_1}
  u_{n-1}(t_{n-1}, X^{n, \bar{\gamma}_{\bar{t}}}(t_{0}), \cdots, X^{n, \bar{\gamma}_{\bar{t}}}(t_{n-1}), X^{n, \bar{\gamma}_{\bar{t}}}(t_{n-1}) ) = Y^{n, \bar{\gamma}_{\bar{t}}}(t_{n-1}).
\end{equation}
Note that this term becomes a terminal condition for the next interval $[t_{n-2}, t_{n-1}]$. Moreover, for a path $\bar{\gamma}_{\bar{t}}$ with $\bar{t} \in [t_{n-1}, t_n]$, the following holds,
\begin{equation}
\begin{aligned}\label{Time terminal condition}
    Y^{n, \bar{\gamma}_{\bar{t}}}(\bar{t}) &= u_{n-1}(\bar{t}, X^{n, \bar{\gamma}_{\bar{t}}}(t_{0}), \cdots, X^{n, \bar{\gamma}_{\bar{t}}}(t_{n-1}), X^{n, \bar{\gamma}_{\bar{t}}}(\bar{t}) )\\
  &= u_{n-1}(\bar{t}, \bar{\gamma}_{\bar{t}}(t_{0}), \cdots, \bar{\gamma}_{\bar{t}}(t_{n-1}), \bar{\gamma}_{\bar{t}}(\bar{t}) ) .
\end{aligned}
\end{equation}

Now, we expand the above argument to the entire interval. For $k = 0, \cdots, n-2$, we inductively define $u_{k}(s, x_{0}, \cdots, x_{k}, y) : [t_{k}, t_{k+1}] \times \mathbb{R}^{k+1} \times \mathbb{R}$ by a solution to the following PDE with parameters $x_{0}, \cdots, x_{k} \in \mathbb{R}$:
\begin{equation} \label{Time pde index k}
\begin{aligned}
  \partial_{t}u_{k}(s, x_{0}, \cdots, x_{k}, y) &+ L^{(n)}_{k}(s, x_{0}, \cdots, x_{k})u_{k}(t, x_{0}, \cdots, x_{k}, y) = 0, \\
  u_{k}(t_{k+1}, x_{0}, \cdots, x_{k-1}, y) &= u_{k+1}(t_{k+1}, x_{0}, \cdots, x_{k-1}, y, y).
\end{aligned}
\end{equation}
For a path $\bar{\gamma}_{\bar{t}} $ with $\bar{t} \in [t_{k}, t_{k+1}]$, we have an equation as in equation \eqref{Time terminal condition}, 
\begin{align}
  u^{(n)}(\bar{\gamma}_{\bar{t}}) = Y^{n, \bar{\gamma}_{\bar{t}}}(\bar{t}) = u_{k}(\bar{t}, \bar{\gamma}_{\bar{t}}(t_{0}), \cdots, \bar{\gamma}_{\bar{t}}(t_{k}), \bar{\gamma}_{\bar{t}}(\bar{t}) ) .
\end{align}


Under the above setting, we define a function $\tilde{u}^{(n)}$ and a differential operator $\tilde{L}^{(n)}$ that are used in \eqref{Tim main idea}. Because the sequence of partitions $\pi^{(n)}$ covers the time $t$, there exists $j \in \{0, 1, \cdots, n-1\}$ such that $t = t_{j}$.
We define a function $\tilde{u}^{(n)} : [t, t+\delta) \times \mathbb{R}^{j+1} \rightarrow \mathbb{R}$ for small $\delta > 0$ and a differential operator as 
\begin{equation} \label{Time def of uLtilde}
\begin{aligned}
  \tilde{u}^{(n)}(v, \eta(t_0), \cdots, \eta(t_j)) &= \sum_{i = j}^{n-1} u_i(v, \eta(t_0), \cdots, \eta(t_j), \eta(t), \cdots, \eta(t)) \mathds{1}_{[t_i, t_{i+1} )}(v) \\
  \tilde{L}^{(n)} (v, \eta(t_0), \cdots, \eta(t_j)) &= \sum_{i=j}^{n-1} L^{(n)}_i(v, \eta(t_0), \cdots, \eta(t_j), \eta(t), \cdots, \eta(t)) \mathds{1}_{( t_i, t_{i+1} )}(v).
\end{aligned} 
\end{equation}
Note that the function $\tilde{u}^{(n)}$ is differentiable with respect to the time  except at points in the partition $\sigma$. 
This definition is used to represent $u^{(n)}(\gamma_{t, t + \delta}) - u^{(n)}(\gamma_{t})$ as a single term.
\begin{lemma}
  With a defined $\tilde{u}^{(n)}$ and a differential operator $\tilde{L}^{(n)}$ in \eqref{Time def of uLtilde}, the following holds:
\begin{equation}
  \begin{aligned}
    u^{(n)}(\gamma_{t, t+\delta}) - u^{(n)}(\gamma_{t}) = \int_{t}^{t+\delta} \partial_{t}\tilde{u}^{(n)}(v, \gamma_{t}(t_0), \cdots, \gamma_{t}(t_j))\,dv
  \end{aligned}
\end{equation}
  The integral of $\partial \tilde{u}^{(n)}$ is defined except at points in the partition.
\end{lemma}
\begin{proof}
  We use the differentiability of $u_{k}$ with respect to $y$ in \eqref{Time def of uLtilde} to obtain 
  $u_{k}(s, x_{0}, \cdots, x_{k}, y) - u_{k}(t_{k}, x_{0}, \cdots, x_{k}, y)=  \int_{t_{k}}^{s} \partial_{t}u_{k}(v, x_{0}, \cdots, x_{k}, y) \,dv$
  for each $k = 0, \cdots, n-1$ and $s \in [t_k, t_{k+1}]$.  
  Suppose $t_j = t$ and $t+\delta \in [t_k, t_{k+1})$ for some $j, k$.
Since the function $u_{i}$ for $ j \le i \le k-1$ satisfies the terminal condition of PDE \eqref{Time pde index k}, we have 
\begin{equation}
  \begin{aligned}
    & u^{(n)}(\gamma_{t, t+\delta}) - u^{(n)}(\gamma_{t})  \\
    = \; &u_{k}(t+\delta, \gamma_{t}(t_0), \cdots, \gamma_{t}(t_{j}), \gamma_{t}(t), \cdots ,  \gamma_{t}(t)) - u_{j}(t, \gamma_{t}(t_0), \cdots, \gamma_{t}(t_j), \gamma_{t}(t_j)) \\
    = \; & \int_{t_{k}}^{t+\delta} \partial_{t}u_{k}( v, \gamma_{t}(t_0), \cdots, \gamma_{t}(t_{j}), \gamma_{t}(t), \cdots ,  \gamma_{t}(t) )\,dv \\
                  &+ \sum_{i=j}^{k-1} \int_{t_{j}}^{t_{j+1}} \partial_{t}u_{i}(v, \gamma_{t}(t_0), \cdots, \gamma_{t}(t_{j}), \gamma_{t}(t), \cdots ,  \gamma_{t}(t)) \,dv \\
                   = \; &\int_{t_j}^{t+\delta} \partial_{t}\tilde{u}^{(n)}(v, \gamma_{t}(t_0), \cdots, \gamma_{t}(t_j))\,dv 
  \end{aligned}
\end{equation}

  The last equation is what we want.
\end{proof}

Now, we can proceed to the last equality in \eqref{Tim main idea}.

\begin{lemma}
  Let $\tilde{L}^{(n)}$, $\tilde{u}^{(n)}$ be as defined in the previous lemma. Suppose that all partition cover the discontinuous points of $\gamma_{t}$ and $t$. 
   For fixed $\delta > 0$, we have
  \begin{equation}
   \begin{aligned}
    \lim_{n \rightarrow \infty} \int_{t}^{t+\delta} -\tilde{L}^{(n)}(v, \gamma_{t}(t_0), \cdots, \gamma_{t}(t_{j(n)}))\tilde{u}^{(n)}(v, \gamma_{t}(t_0), \cdots, \gamma_{t}(t_{j(n)}))\,dv = - \int_{t}^{t+\delta}L(\gamma_{t, v})u_{v}(\gamma_{t, v})\,dv,
  \end{aligned}
\end{equation}
  where $t_{j(n)} \in \pi^{(n)}$ satisfy $t_{j(n)} = t$ for each $n \in \mathbb{N}$.
\end{lemma}

\begin{proof}
    Let $t_{k(n)} \in \pi^{(n)}$ satisfy $v \in [t_{k(n)}, t_{k(n)+1})$. From the definition of $\gamma_{t, v}$, we have 
    \begin{equation}
\begin{aligned}
  &u_{k(n)}(v,\gamma_{t}(t_0), \cdots, \gamma_{t}(t_{j(n)}), \gamma_{t}(t), \cdots, \gamma_{t}(t) ) \\
  = \; & u_{k(n)}(v, \gamma_{t, v}(t_{0}), \cdots, \gamma_{t, v}(t_{k(n)}), \gamma_{t, v}(v) ) = Y^{n, \gamma_{t, v}}(v) = u^{(n)}(\gamma_{t, v}).
\end{aligned}
\end{equation}
    Then, we can understand the partial derivatives as a vertical derivative of a solution $Y^{n, \gamma_{t, v}}(v)$, since it holds that
\begin{equation}
  \frac{\partial u_{k(n)}}{\partial x_{k(n)}}(v,\gamma_{t}(t_0), \cdots, \gamma_{t}(t_{j(n)}), \gamma_{t}(t), \cdots, \gamma_{t}(t) )
  =\lim_{h \rightarrow 0 } \frac{u^{(n)} (\gamma_{t, v}^{h}) - u^{(n)}(\gamma_{t, v}) }{h} = D_{x}u^{(n)}(\gamma_{t, v}).
\end{equation}
 Using \eqref{Time def of uLtilde} and \eqref{Time def of bnk} we have 
\begin{equation}\label{Tim lem9 eq1}
\begin{aligned}
  &-\tilde{L}^{(n)}(v, \gamma_{t}(t_0), \cdots, \gamma_{t}(t_{j(n)}))\tilde{u}^{(n)}(v, \gamma_{t}(t_0), \cdots, \gamma_{t}(t_{j(n)}))  \\
   = \; & - b_{v}( \varphi^{n}_{L}(\gamma_{t, v}) ) D_xu^{(n)}(\gamma_{t, v}) - \frac{1}{2}(\sigma_{v})^{2}(\varphi^{n}_{L}(\gamma_{t, v}))D_{xx}u^{(n)}(\gamma_{t, v}).
\end{aligned}
\end{equation}
    for $v \in [t, t+\delta]$.
Note that the $b_{v}( \varphi^{n}_{L}(\gamma_{t, v}) )$ converges to $b_v (\gamma_{t, v})$ in a sup-norm as $n \rightarrow \infty$. The case of $\sigma$ is same. 
 From Lemma~\ref{Time u^n Dxu^n} and linear growth of $b$ and $\sigma$, the dominating convergence theorem gives what we want.    
\end{proof}

Now, we are ready to prove Theorem~\ref{Time thm1}.
\begin{proof} [Proof of Theorem~\ref{Time thm1}]
   To conclude \eqref{Time aim}
   using the fundamental theorem of calculus from equations \eqref{Tim main idea}, we only need to show that $Lu_{s}(\gamma_{t, s})$ is continuous with respect to $s > t$. 
  Since the non-anticipative functionals $b$, $\sigma$ are functional Lipschitz continuous, it is sufficient to show the continuity of $D_{x}u_s(\gamma_{t,s})$ and $ D_{xx}u_s(\gamma_{t,s})$.
  We give the proof only in the case in which $D_{x}u_s(\gamma_{t,s})$; the other cases are similar. 
  For $s > t$, we have 
\begin{equation}
  \begin{aligned} \label{Tim ineq in pf of main thm}
    \left\lvert  D_{x}u_s(\gamma_{t,s}) - D_{x}u_t(\gamma_t) \right\rvert 
    &\le  \left\lvert   D_{x}Y^{\gamma_{t, s}}(s) - D_{x}Y^{\gamma_{t}}(s)  \right\rvert + \left\lvert D_{x}Y^{\gamma_{t}}(s) -  D_{x}Y^{\gamma_{t}}(t) \right\rvert.
  \end{aligned}
  \end{equation}
  Because $D_{x}Y^{\gamma_{t}}$ is the solution to BSDE \eqref{Ver dyn of DY in thm2}, and $D_{x}Y^{\gamma_{t}}(v)$ is continuous for $v \ge t$ almost everywhere. 
  Thus, the second term goes to $0$ as $s \rightarrow t+$. 
  The first term can be estimated as in the argument in $(iii)$ Lemma~\ref{Ver lem1} from the functional Lipschitz continuity of $Dg$.
\end{proof}

We have proven the existence of the vertical and horizontal derivatives of non-anticipative functional $u$. Now, we describe the proof of Theorem~\ref{main theorem}. 

\begin{proof} [Proof of Theorem~\ref{main theorem}] \label{proof of main theorem}
From the argument in Section~\ref{sec:2}, recall that we may regard the non-anticipative function $F$ as an operator $\tilde{F}$ from $\tilde{D}$ to $\mathbb{R}$. Conversely, the operator $\tilde{F}$ from $\tilde{D}$ to $\mathbb{R}$ can be regard as a non-anticipative function by $F(t, \gamma) = \tilde{F}(\gamma^{t})$. Therefore, we can regard $u$ defined in \eqref{main def of u} as a non-anticipative functional. 

The existences of $D_{x}u$, $D_{xx}u$, $D_{t}u$ are derived from Theorems~\ref{Ver thm2} and \ref{Time thm1}. We now turn to the converse direction. The operator $u : \tilde{D} \rightarrow \mathbb{R}$ corresponds to the non-anticipative functional $u : [0, T] \times D([0, T], \mathbb{R}) \rightarrow \mathbb{R}$ as $\tilde{u}(t, \gamma) = u(\gamma)$ for $\gamma \in D([0, t], \mathbb{R})$.
  From the functional It\^{o} formula, we have 
$d\tilde{u}(s, X^{\gamma_{t}}_{T}) = D_{x}\tilde{u}(s, X^{\gamma_{t}}_{s}) \sigma_{s}(X^{\gamma_{t}}) \,dW_s$.
Integrating both sides on $[t, T]$, we obtain 
\begin{equation}
\tilde{u}(t, X^{\gamma_{t}}_{t}) = \tilde{u}(T, X^{\gamma_{t}}_{T}) -\int_{t}^{T} D_{x}\tilde{u}(s, X^{\gamma_{t}}_{s}) \sigma_{s}(X^{\gamma_{t}}) \,dW_s
\end{equation}
with $X^{\gamma_{t}}_T \in D([0, t], \mathbb{R})$. From the terminal condition of \eqref{Tim ppde in thm1}, we know that $\tilde{u}(T, X^{\gamma_{t}}_{T}) = g(X^{\gamma_{t}})$. Since the solution to the BSDE is unique, we obtain the desired result.
\end{proof}
Finally, we prove Proposition~\ref{greek value prop1}, which states the self-financing portfolio of options.
\begin{proof}[Proof of Proposition~\ref{greek value prop1}]
  We apply the It\^{o} formula to $v(X_t)$. Since $u_t(X_t)$ is a solution to PPDE \eqref{main PPDE_option}, we have
  \begin{equation}
  \begin{aligned}
    du_t(X_t)  
             & = D_{x}u_t(X_t)\sigma_{t}(x)\,dW_{t} = -D_{x}u_t(X_t) (r - q_{t}(X)) \,dt + \frac{D_{x}u_t(X_t)}{S_{t}} \,dS_{t}
  \end{aligned}
\end{equation}
  using the definition of $S_{t}$. Thus, we obtain that
\begin{equation}
  \begin{aligned}
    dv_t(X_t) = rv_{t}(X_t) - D_{x}u_t(X_t)(r - q_t(X)) e^{-r(T-t)}  + \frac{D_{x}u_t(X_t)}{S_{t}}e^{-r(T-t)} \,dS_{t}.
  \end{aligned}
  \end{equation}
  This means that $v_t(X_t)$ is a self-financing portfolio with initial value $u_0(X_0)e^{-rT}$ with position $ \frac{D_{x}u_t(X_t)}{S_{t}}e^{-r(T-t)}$. Finally, using the martingale property of $u_t(X_t)$, we obtain $\mathbb{E}[u_T(X_T)] = u_0(X_0)$.
\end{proof}

\section{Greeks}\label{Appendix:Greeks}
In this section, we prove the propositions and theorem in Section~\ref{subsec:G}. 
We follow the setting in Section~\ref{subsec:G}.
Under the risk-neutral setting, we define the stock model $S_t$ and logarithm $X_t$ described in Section~\ref{subsec:PSDE}. 
We consider families of coefficients  $(r^{\epsilon})_{\epsilon \in I},$ 
$(q_\cdot^{\epsilon})_{\epsilon \in I}$ and $(\sigma_\cdot^{\epsilon})_{\epsilon \in I}$
where $I=(-1,1)$
and $\epsilon$ is the perturbation parameter.
For convenience, we define
$ b^{\epsilon}_t :=r^{\epsilon } -q_t^\epsilon- \frac{1}{2}(\sigma^{\epsilon}_t)^{2}.$
We also assume that families of coefficient $b^{\epsilon}$ and $\sigma^{\epsilon}$ satisfy Assumption~\ref{greek assum}.

First, we mention that $X$ and $X^{\epsilon}$ are actually close in $S^{p}$ space.
\begin{lemma} \label{greeks_app lem1}
  For $\epsilon$ sufficiently close to $0$ and $p \ge 2$, we have
  \begin{align}
    \mathbb{E}\Big[  \sup_{u \in [0, T]}\left\lvert X(u) - X^{\epsilon}(u) \right\rvert^{p}   \Big] \le C \epsilon^{p}.
  \end{align}
\end{lemma}
The proof is the same as in Theorem~B.3 of \cite{fournie2010functional}. Now, let us prove Proposition~\ref{greek value prop1}.

\begin{proof} [Proof of Proposition~\ref{greek value prop2}]
  Applying the functional It\^{o} formula Theorem~\ref{ftnl Ito formula} to $u_t(X^{\epsilon}_t)$, we obtain
  \begin{equation}
  \begin{aligned}
    &du_t(X^{\epsilon}_t) = \big(  D_{x}u_t(X^{\epsilon}_t)(b^{\epsilon}_{t}(X^{\epsilon}) - b_{t}(X^{\epsilon}) ) + \frac{1}{2}((\sigma^{\epsilon}_{t})^{2}(X^{\epsilon}) - \sigma^{2}_{t}(X^{\epsilon})) D_{xx}u_t(X^{\epsilon}_t) \big)\,dt + D_{x}u_t(X^{\epsilon}_t)\sigma^{\epsilon}_{t}(X^{\epsilon})\,dW_{t}.
  \end{aligned}
\end{equation}
  We use the fact that $u$ is the solution to PPDE \eqref{main PPDE} in the third equality. Therefore, we obtain
  \begin{equation}
  \begin{aligned}
    dG^{\epsilon}(t)  = D_{x}u_t(X^{\epsilon}_t)\sigma^{\epsilon}_{t}(X^{\epsilon})\,dW_{t} = - (r^{\epsilon} - q^{\epsilon}_t(X^{\epsilon}))D_{x}u_t(X^{\epsilon}_t)\,dt + \frac{D_{x}u_t(X^{\epsilon}_t)}{S^{\epsilon}_{t}} \,dS^{\epsilon}_{t}.
  \end{aligned}
  \end{equation}
  Therefore, as in the proof of Proposition~\ref{greek value prop1}, we obtain what we want.
\end{proof}

Now, for each $\epsilon>0$, define  $A_t = G^{\epsilon}(t) - u_t(X^{\epsilon}_t)$
Then, for each $\epsilon > 0$, we get equalities from the martingale property of $u_t(X_t)$ and $G^{\epsilon}$
\begin{equation}
\begin{aligned}
  &v^{\epsilon}_0(X_{0}) - v_0(X_0) = \mathbb{E}[ u_T(X^{\epsilon}_T)e^{-r^{\epsilon}T} - u_T(X_T)e^{-rT}] \\
  = \; & \mathbb{E}[G^{\epsilon}(T) - A_T]e^{-r^{\epsilon}T} - u_0(X_{0})e^{-rT}  = -\mathbb{E}[A_T]e^{-r^{\epsilon}T} + u_0(X_{0}) ( e^{-r^{\epsilon}T}  - e^{-rT}).
\end{aligned}
\end{equation}
Using the above argument, we prove the sensitivity formula of Theorem~\ref{greek formula thm}.

\begin{proof} [Proof of Theorem~\ref{greek formula thm}]
  We have shown that 
  \begin{equation}
  \begin{aligned}
    \lim_{\epsilon \rightarrow 0} \frac{1}{\epsilon} (v^{\epsilon}_0(X_{0}) - v_0(X_0) )
    = \lim_{\epsilon \rightarrow 0}  \Big(  u_0(X_{0})\frac{e^{-r^{\epsilon}T} - e^{-rT}}{\epsilon}  - \mathbb{E}[\frac{A_T}{\epsilon}]e^{-r^{\epsilon}T} \Big).
  \end{aligned}
  \end{equation}
  Note that for $\psi = b, \sigma, q$, we have $\psi^{\epsilon}(\eta) \rightarrow \psi(\eta)$ and $\frac{\psi^{\epsilon}_{t}(\eta) - \psi_{t}(\eta)}{\epsilon} \rightarrow \dot{\psi}(\eta)$ as $\epsilon \rightarrow 0$ for each $\eta \in D([0, T], \mathbb{R})$. Then, $\dot{r}$ comes from the formula of $r$ and $r^{\epsilon}$ as $r = b_{t} + \frac{1}{2}\sigma_{t}^{2} + q_t, r^{\epsilon} = b^{\epsilon}_{t} + \frac{1}{2}(\sigma^{\epsilon}_{t})^{2} + q^{\epsilon}_t. $
  Thus, we obtain that $\dot{r} = \dot{b}_{t} + \dot{\sigma}_{t}\sigma_{t} + \dot{q}_t$.
  
  For the term $\mathbb{E}[A_T/\epsilon]$, we sequentially use the triangle inequality and apply the Lebesgue dominating convergence theorem to each term. In the process, we use Theorem~\ref{Ver thm1}, Lemma~\ref{greeks_app lem1} and the Lipschitz continuity of $b$ and $\sigma$. 

  The fractional part in $\mathbb{E}[A_T/\epsilon]$ can be regarded as differential by showing
  $ \mathbb{E}[  \int_{0}^{T} \lvert D_{x}u_{s}(X^{\epsilon}_s) ( (b^{\epsilon}_{s}(X^{\epsilon}) - b_{s}(X^{\epsilon}))/ \epsilon - \dot{b}_{s}(X^{\epsilon}) ) \rvert  \, ds  ] $ goes to $0$. Then, we can omit $\epsilon$ of $X$  as  
$\mathbb{E} [  \int_{0}^{T} \lvert D_{x}u_{s}(X^{\epsilon}_s) ( \dot{b}_{s}(X^{\epsilon}) - \dot{b}_{s}(X) ) \rvert \, ds ]$
and  $\mathbb{E}[\int_{0}^{T} \lvert (D_{x}u_{s}(X^{\epsilon}_s) - D_{x}u_s(X_s)) \dot{b}_{s}(X) \rvert \,ds ]$ converges to 0.
Consequently, we obtain that
  \begin{align}
     \mathbb{E}\left[  \int_{0}^{T} \left\lvert D_{x}u_{s}(X^{\epsilon}_s) \left( \frac{b^{\epsilon}_{s}(X^{\epsilon}) - b_{s}(X^{\epsilon}) }{\epsilon} \right)  - D_{x}u_s(X_s) \dot{b}_{s}(X) \right\rvert \,ds \right] \rightarrow 0, \quad \textnormal{as} \quad \epsilon \rightarrow 0.
  \end{align}
  Similarly, we can estimate the remaining term of $\sigma$. Therefore, we conclude that 
  \begin{align}
    \lim_{\epsilon \rightarrow 0} \mathbb{E}\left[ \frac{A_T}{\epsilon} \right]  e^{-r^{\epsilon}T} =  \mathbb{E}\left[\int_{0}^{T} D_{x}u_s(X_s) \dot{b}_{s}(X) +  D_{xx}u_s(X_s) \dot{\sigma}_{s}(X)\sigma_{s}(X) \,ds\right]e^{-rT}.
  \end{align}
\end{proof}

\end{appendices}

%
%

\bibliographystyle{spbasic}      
\bibliography{./Option_pricing_under_path-dependent_stock_models.bib}   

\end{document}